\documentclass[10pt]{article}%
\usepackage{amsmath}
\usepackage{amsfonts}
\usepackage{mathrsfs}
\usepackage{amssymb,color}
\usepackage[linkcolor=black,anchorcolor=black,citecolor=black]{hyperref}
\usepackage{graphicx}
\numberwithin{equation}{section}
\usepackage[body={15.5cm,21cm}, top=3cm]{geometry}%
\providecommand{\U}[1]{\protect\rule{.1in}{.1in}}
\providecommand{\U}[1]{\protect \rule{.1in}{.1in}}
\newtheorem{theorem}{Theorem}[section]

\newtheorem{lemma}[theorem]{Lemma}

\newtheorem{proposition}[theorem]{Proposition}
\newtheorem{remark}[theorem]{Remark}

\newenvironment{proof}[1][Proof]{\noindent \textbf{#1.} }{\  \rule{0.5em}{0.5em}}

\begin{document}
	\title{Backward Stochastic Differential Equations Driven by $G$-Brownian Motion with Double Reflections}
	\author{ Hanwu Li\thanks{Center for Mathematical Economics, Bielefeld University, Bielefeld 33615, Germany. E-mail:
			hanwu.li@uni-bielefeld.de.} \and Yongsheng Song\thanks{RCSDS, Academy of Mathematics and Systems Science, Chinese Academy of Sciences, Beijing 100190, China, and
School of Mathematical Sciences, University of Chinese Academy of Sciences, Beijing 100049, China. E-mail:
yssong@amss.ac.cn.}}
	\date{}
	\maketitle
	\begin{abstract}
	In this paper, we study the reflected  backward stochastic differential equations driven by $G$-Brownian motion with two reflecting obstacles, which means that the solution lies between two prescribed processes. A new kind of \emph{approximate Skorohod condition} is proposed to derive the uniqueness and existence of the solutions. The uniqueness can be proved by a priori estimates and the existence is obtained via a penalization method.
	\end{abstract}
	
	\textbf{Key words}: $G$-expectation, reflected backward SDE, approximate Skorohod condition
	
	\textbf{MSC-classification}: 60H10
	
	\section{Introduction}
	Given a filtered probability space $(\Omega,\mathcal{F},(\mathcal{F}_t)_{t\in[0,T]},P)$, Pardoux and Peng \cite{PP} first introduced the following type of nonlinear backward stochastic differential equations (BSDEs for short):
	\begin{displaymath}
	  Y_t=\xi+\int_t^T f(s,Y_s,Z_s)ds-\int_t^T Z_sdB_s,
	\end{displaymath}
	where the generator $f(\cdot,y,z)$ is progressively measurable and Lipschitz continuous with respect to $(y,z)$, $\xi$ is an $\mathcal{F}_T$-measurable and square integrable terminal value. They proved that there exists a unique pair of progressively measurable processes $(Y,Z)$ satisfying this equation. The BSDE theory attracts a great deal of attention due to its wide applications in mathematical finance, stochastic control and quasilinear partial differential equations (see \cite{EPQ97}, \cite{PP92}, etc).
	
	One of the most important extensions is the reflected BSDE initiated by El Karoui, Kapoudjian, Pardoux, Peng and Quenez \cite{KKPPQ}. In addition to the generator $f$ and the terminal value $\xi$, there is an additional continuous process $S$, called the obstacle, prescribed in this problem. The reflection means that the solution is forced to be above this given process $S$. More precisely, the solution of the reflected BSDE with parameters $(\xi,f,S)$ is a triple of processes $(Y,Z,L)$ such that
	\begin{align*}
		&Y_t=\xi+\int_t^T f(s,Y_s,Z_s)ds+L_T-L_t-\int_t^T Z_sdB_s,\\
		&Y_t\geq S_t, \ t\in[0,T], \textrm{ and } \int_0^T(Y_s-S_s)dL_s=0, P\textrm{-a.s.},
	\end{align*}
	where $L$ is an increasing process to push the solution upwards. Besides, it should behave in a minimal way, which means that $L$ only acts when the solution $Y$ reaches the obstacle $S$. This requirement corresponds to the mathematical expression
	$ \int_0^T(Y_s-S_s)dL_s=0$, called Skorohod condition. The reflected BSDE is a useful tool to study problems of pricing American options, the obstacle problem for quasilinear PDEs as well as the variational inequalities (see \cite{BCFE}, \cite{KKPPQ}, \cite{EPQ}).
	
	Building upon these results, Cvitanic and Karaztas \cite{CK} studied BSDEs with two reflecting obstacles, which means that the solution $Y$ is forced to stay between a lower obstacle $L$ and an upper obstacle $U$. This can be achieved by the combined actions of two increasing processes: one is to push the solution upwards, the other is to push it downwards and both of them act in a minimal way when $Y$ tries to cross the obstacles.  They also established the relation between the solution of the doubly reflected BSDE and the value function of Dynkin game. For more details about this topic, we refer to the papers \cite{CM,DQS,GIOQ,HL,HLM,PX}.
	
	Note that the classical BSDE and reflected BSDE theory can only deal with financial problems under mean uncertainty, not volatility uncertainty, and can give probabilistic interpretations for quasi linear PDEs, not fully nonlinear ones. Motivated by these facts, Peng (\cite {P07a}, \cite {P08a}) systematically established the $G$-expectation theory. A new type of Brownian motion $B$, called $G$-Brownian motion, whose increments are stationary and independent, was constructed. Different from the classical case, the quadratic variation process $\langle B\rangle$ is not deterministic. The basic notions and tools, such as the stochastic integral with respect to $G$-Brownian motion $B$ and $G$-It\^{o}'s formula, were also established.
	
	A few years later, Hu, Ji, Peng and Song \cite{HJPS1} established the well-posedness of BSDEs driven by $G$-Brownian motion ($G$-BSDEs for short) as the following:
	\begin{displaymath}
		Y_t=\xi+\int_t^T f(s,Y_s,Z_s)ds+\int_t^T g(s,Y_s,Z_s)d\langle B\rangle_s-\int_t^T Z_sdB_s-(K_T-K_t),
	\end{displaymath}
	where the generators $f$, $g$ are Lipschitz continuous with respect to $(Y,Z)$. Under conditions similar to the classical case, applying the Galerkin approximation technique and the PDE approach, they proved that there exists a unique solution $(Y,Z,K)$ to this equation, where $K$ is a decreasing $G$-martingale. Besides, in the accompanying paper \cite{HJPS2}, they also obtained the comparison theorem, Girsanov transformation and the nonlinear Feynman-Kac formula.
	
	Li, Peng and Soumana Hima \cite{LPSH} first studied the reflected $G$-BSDE with a lower obstacle. Due to the appearance of the decreasing $G$-martingale, the Skorohod condition was replaced by a martingale condition in order to get the uniqueness of the solutions. The existence was proved by the approximation method via penalization. Li and Peng \cite{lp} also considered the upper obstacle case.  However, in order to pull the solution down below the upper obstacle, one needs to add a decreasing process $L$ in the $G$-BSDE. Hence, the main difficulty is that the process $L-K$ is not monotonic as in the lower obstacle case. Although they did not  obtain the uniqueness, they showed that the solution constructed by a penalization method is a maximal one by a variant comparison theorem.
	
	In this paper, we investigate the doubly reflected BSDE driven by $G$-Brownian motion with two obstacles $(L,U)$. As in the classical case, there should be two increasing processes $A^+, A^-$: one aims to push the solution upward while the other is to pull the solution downward, and both processes behave in a minimal way such that they satisfy the Skorohod condition. Besides, there will also be a decreasing $G$-martingale $K$ as in the $G$-BSDE, which exhibits the uncertainty of the model. Therefore, it is natural to conjecture that a solution to this doubly reflected $G$-BSDE should be  a 5-tuples of processes $(Y, Z, K, A^+, A^-)$ with $L_t\leq Y_t\leq U_t$ satisfying
\begin{align*}
		&Y_t=\xi+\int_t^T f(s,Y_s,Z_s)ds-\int_t^T Z_sdB_s-(K_T-K_t)+(A_T^{+}-A_t^{+})-(A_T^{-}-A_t^{-}),\\
		&\int_0^T (Y_s-L_s)dA^+_s=\int_0^T (U_s-Y_s)dA^-_s=0.
	\end{align*}
However,  the processes $A^+, \  A^-$ and $K$ here are mixed together,  and the above Skorohod condition is not applicable. In this paper, we write $A$ for $A^+-A^--K$ and replace the Skorohod condition by a new kind of  \textsl{Approximate Skorohod Condition}, which turns into the martingale condition when there is only one obstacle.

The uniqueness of the solutions is obtained by a priori estimates requiring some delicate analysis.  In order to prove the existence, we consider the following $G$-BSDEs parameterized by $n=1,2,\cdots$,
	\begin{displaymath}
		Y_t^n=\xi+\int_t^T f(s,Y_s^n,Z^n_s)ds-\int_t^T Z_s^ndB_s-(K_T^n-K_t^n)+(A_T^{n,+}-A_t^{n,+})-(A_T^{n,-}-A_t^{n,-}),
	\end{displaymath}
	where $A_t^{n,+}=\int_0^t n(Y_s^n-L_s)^-ds$, $A_t^{n,-}=\int_0^t n(Y_s^n-U_s)^+ds$.
	
 The objective, similar to the classical case studied by Cvitanic and Karaztas \cite{CK},  is to show that the sequence $(Y^n,Z^n,A^n)$, where $A^n=A^{n,+}-A^{n,-}-K^n$, converges to a triple of processes $(Y,Z,A)$, and that $(Y,Z,A)$ is a  solution to the doubly reflected $G$-BSDE.  To this end,  the dominated convergence theorem and the property of weakly compactness played crucial role in  Cvitanic and Karaztas \cite{CK}.  However,  these tools are not available under the $G$-expectation framework.

 Our proof is divided into two stages.

 \textbf{Stage 1.} We establish the uniform estimates for $Y^n$ under the norm $\|\cdot\|_{S_G^\alpha}$,  and  prove that $(Y^n-U)^+$ and $(Y^n-L)^-$ converge to $0$ under the norm $\|\cdot\|_{S_G^\alpha}$.  These properties hold true under the assumption that the upper and lower obstacles belong to the space $S_G^\beta(0,T)$ and they are separated by some generalized $G$-It\^{o} process (see (A3').  The latter implies that the limit $Y$ (if exists) lies between the upper and lower obstacles.

  \textbf{Stage 2.} We show that the sequences $A^{n,+}_T$, $A^{n,-}_T$, $K^n_T$ (resp. $Z^n$) are uniformly bounded under the norm $\|\cdot\|_{L_G^\alpha}$ (resp. $\|\cdot\|_{H^{\alpha}_G}$). For this purpose, we prove that $(Y^n-U)^+$ converges to $0$ with the explicit rate $\frac{1}{n}$, which requires that the upper obstacle is a generalized $G$-It\^{o} process.

  Based on the above analysis, we obtain the convergence of $(Y^n,Z^n,A^n)$, and consequently the existence of the doubly reflected $G$-BSDE.

Recall that, the $G$-expectation can be represented as the supremum of the linear expectation under the probability $P$ over all $P\in\mathcal{P}$, where $\mathcal{P}$ is a collection of mutually singular martingale measures.  Therefore, the $G$-expectation theory shares many similarities with the quasi-sure analysis by Denis and Martini \cite{DM} and the second order BSDEs by Soner, Touzi and Zhang \cite{STZ13} and Matoussi, Possama\"{i} and Zhou \cite{MPZ}.  Compared with these works, one advantage of the $G$-expectation framework is that the solution to the $G$-BSDEs is a (generalized) $G$-It\^o process,  and that the decomposition of (generalized) $G$-It\^o processes is unique. This amounts to say that the derivatives $\partial_tu$, $\partial_xu$ and the second order derivative $\partial^2_xu$ of a function $u(t,x)$ are all well defined in the $G$-expectation space, which is crucial to give the probabilistic representations for (path dependent) fully nonlinear PDEs.  In other words, the solutions of $G$-BSDEs have strong regularity and can be universally defined in the spaces of the $G$-framework, which enhances the results in \cite{STZ13} and \cite{MPZ}.

The problem considered in this paper is closely related to Matoussi, Piozin and Possama\"{i} \cite{MPP}, which  studied the second order BSDEs with general reflections, but it is formulated in a quite different way.

1) The solution $(Y,Z, A)$ to the doubly reflected $G$-BSDE is defined in the $G$-framework, in which the processes have strong regularity and remarkable properties. As is mentioned above, in the $G$-framework, the unique decomposition of It\^o processes implies that the derivative $\partial^2_xu$ is well defined, which embodies the advantages of the $G$-expectation compared to the linear expectations.

2) In \cite{MPP} and the corrigendum \cite{MPZ'}, the process $V$ (corresponding to the process $A$ in this paper) is defined and characterized by the Skorohod condition individually for each probability $P$ in $\mathcal {P}$. In this paper, the process $A$ and the corresponding approximate Skorohod condition are given universally with respect to all probabilities $P$ in $\mathcal {P}$. 

    This paper is organized as follows. In Section 2, we present some notions and results on $G$-expectation and $G$-BSDEs as preliminaries. In Section 3, we first state the definition of solution to doubly reflected $G$-BSDE and establish some a priori estimates from which we can derive the uniqueness of the solution. We then introduce the penalization method to prove the existence of the solution in Section 4. 

	\section{Preliminaries}
	In this section, we review notations and results in the $G$-expectation framework, which are concerned with the $G$-It\^{o} calculus and BSDE driven by $G$-Brownian motion. For simplicity, we only consider the one-dimensional case. For more details, we refer to the papers  \cite{HJPS1}, \cite{HJPS2},
	\cite{P07a}, \cite{P08a}, \cite{P10}.
	
	Let $\Omega=C_{0}([0,\infty);\mathbb{R})$, the space of
	real-valued continuous functions starting from the origin, be endowed
	with following norm,
	\begin{displaymath}
		\rho(\omega^1,\omega^2):=\sum_{i=1}^\infty 2^{-i}[(\max_{t\in[0,i]}|\omega_t^1-\omega_t^2|)\wedge 1], \textrm{ for } \omega^1,\omega^2\in\Omega.
	\end{displaymath}
	
	Let  $B$ be the canonical
	process on $\Omega$. Set
	\[
	L_{ip} (\Omega):=\{ \varphi(B_{t_{1}},...,B_{t_{n}}):  \ n\in\mathbb {N}, \ t_{1}
	,\cdots, t_{n}\in\lbrack0,\infty), \ \varphi\in C_{b,Lip}(\mathbb{R}^{ n})\},
	\]
	where $C_{b,Lip}(\mathbb{R}^{ n})$ denotes the set of bounded Lipschitz functions on $\mathbb{R}^{n}$.  Let $(\Omega,L_{ip}(\Omega),\hat{\mathbb{E}})$ be the $G$-expectation space, where the function $G:\mathbb{R}\rightarrow\mathbb{R}$ is defined by
	\begin{displaymath}
		G(a):=\frac{1}{2}\hat{\mathbb{E}}[aB_1^2]=\frac{1}{2}(\bar{\sigma}^2a^+-\underline{\sigma}^2a^-).
	\end{displaymath}
 In this paper, we always assume that $G$ is non-degenerate, i.e., $\underline{\sigma}^2 >0$. In fact, the (conditional) $G$-expectation for $\xi\in L_{ip}(\Omega)$ can be calculated as follows. Assume that $\xi$ can be represented as
    \begin{displaymath}
    	\xi=\varphi(B_{{t_1}}, B_{t_2},\cdots,B_{t_n}).
\end{displaymath}
    Then, for $t\in[t_{k-1},t_k)$, $k=1,\cdots,n$,
\begin{displaymath}
	\hat{\mathbb{E}}_{t}[\varphi(B_{{t_1}}, B_{t_2},\cdots,B_{t_n})]=u_k(t, B_t;B_{t_1},\cdots,B_{t_{k-1}}),
\end{displaymath}
where, for any $k=1,\cdots,n$, $u_k(t,x;x_1,\cdots,x_{k-1})$ is a function of $(t,x)$ parameterized by $(x_1,\cdots,x_{k-1})$ such that it solves the following fully nonlinear PDE defined on $[t_{k-1},t_k)\times\mathbb{R}$:
\begin{displaymath}
	\partial_t u_k+G(\partial_x^2 u_k)=0
\end{displaymath}
with terminal conditions
\begin{displaymath}
	u_k(t_k,x;x_1,\cdots,x_{k-1})=u_{k+1}(t_k,x;x_1,\cdots,x_{k-1},x), \ k<n
\end{displaymath}
and $u_n(t_n,x;x_1,\cdots,x_{n-1})=\varphi(x_1,\cdots,x_{n-1},x)$. Hence, the $G$-expectation of $\xi$ is $\hat{\mathbb{E}}_0[\xi]$.
	
	For each $p\geq1$,   the completion of $L_{ip} (\Omega)$ under the norm $\Vert\xi\Vert_{L_{G}^{p}}:=(\hat{\mathbb{E}}[|\xi|^{p}])^{1/p}$ is denoted by $L_{G}^{p}(\Omega)$.  The conditional $G$-expectation $\mathbb{\hat{E}}_{t}[\cdot]$ can be
	extended continuously to the completion $L_{G}^{p}(\Omega)$. The canonical process $B$ is the 1-dimensional $G$-Brownian motion in this space.
		
	For each fixed $T\geq 0$, set $\Omega_T=\{\omega_{\cdot\wedge T}:\omega\in \Omega\}$. We may define $L_{ip}(\Omega_T)$ and $L_G^p(\Omega_T)$ similarly.  Besides, Denis, Hu and Peng \cite{DHP11} proved that the $G$-expectation has the following representation.
	\begin{theorem}[\cite{DHP11}]
		\label{the1.1}  There exists a weakly compact set
		$\mathcal{P}$ of probability
		measures on $(\Omega,\mathcal{B}(\Omega))$, such that
		\[
		\hat{\mathbb{E}}[\xi]=\sup_{P\in\mathcal{P}}E_{P}[\xi] \text{ for all } \xi\in  {L}_{G}^{1}{(\Omega)}.
		\]
		$\mathcal{P}$ is called a set that represents $\hat{\mathbb{E}}$.
	\end{theorem}
	
	Let $\mathcal{P}$ be a weakly compact set that represents $\hat{\mathbb{E}}$.
	For this $\mathcal{P}$, we define the capacity%
	\[
	c(A):=\sup_{P\in\mathcal{P}}P(A),\ A\in\mathcal{B}(\Omega).
	\]
	A set $A\in\mathcal{B}(\Omega_T)$ is called polar if $c(A)=0$.  A
		property holds $``quasi$-$surely"$ (q.s.) if it holds outside a
		polar set. In the following, we do not distinguish the two random variables $X$ and $Y$ if $X=Y$, q.s.
	
	For $\xi\in L_{ip}(\Omega_T)$, let $\mathcal{E}(\xi)=\hat{\mathbb{E}}[\sup_{t\in[0,T]}\hat{\mathbb{E}}_t[\xi]]$ and  $\mathcal{E}$ is called the $G$-evaluation. For $p\geq 1$ and $\xi\in L_{ip}(\Omega_T)$, define $\|\xi\|_{p,\mathcal{E}}=[\mathcal{E}(|\xi|^p)]^{1/p}$ and denote by $L_{\mathcal{E}}^p(\Omega_T)$ the completion of $L_{ip}(\Omega_T)$ under $\|\cdot\|_{p,\mathcal{E}}$. The following theorem can be regarded as  Doob's maximal inequality under $G$-expectation.
	\begin{theorem}[\cite{S11}]\label{the1.2}
		For any $\alpha\geq 1$ and $\delta>0$, $L_G^{\alpha+\delta}(\Omega_T)\subset L_{\mathcal{E}}^{\alpha}(\Omega_T)$. More precisely, for any $1<\gamma<\beta:=(\alpha+\delta)/\alpha$, $\gamma\leq 2$, we have
		\begin{displaymath}
		\|\xi\|_{\alpha,\mathcal{E}}^{\alpha}\leq \gamma^*\{\|\xi\|_{L_G^{\alpha+\delta}}^{\alpha}+14^{1/\gamma}
		C_{\beta/\gamma}\|\xi\|_{L_G^{\alpha+\delta}}^{(\alpha+\delta)/\gamma}\},\quad \forall \xi\in L_{ip}(\Omega_T),
		\end{displaymath}
		where $C_{\beta/\gamma}=\sum_{i=1}^\infty i^{-\beta/\gamma}$, $\gamma^*=\gamma/(\gamma-1)$.
	\end{theorem}

    For $T>0$ and $p\geq 1$, the following spaces will be frequently used in this paper.
    \begin{itemize}
    	\item $M_G^0(0,T):=\{\eta: \eta_{t}(\omega)=\sum_{j=0}^{N-1}\xi_{j}(\omega)\textbf{1}_{[t_{j},t_{j+1})}(t),$ where  $\xi_j\in L_{ip}(\Omega_{t_j})$, $t_0\leq \cdots\leq t_N$ is a partition of $[0,T]\}$;
    	\item $M_G^p(0,T)$ is the completion of $M_G^0(0,T)$ under the norm $\Vert\eta\Vert_{M_{G}^{p}}:=(\mathbb{\hat{E}}[\int_{0}^{T}|\eta_{s}|^{p}ds])^{1/p}$;
    	\item $H_G^p(0,T)$ is the completion of $M_G^0(0,T)$ under the norm$\|\eta\|_{H_G^p}:=\{\hat{\mathbb{E}}[(\int_0^T|\eta_s|^2ds)^{p/2}]\}^{1/p}$;
    	\item $S_G^0(0,T)=\{h(t,B_{t_1\wedge t}, \ldots,B_{t_n\wedge t}):t_1,\ldots,t_n\in[0,T],h\in C_{b,Lip}(\mathbb{R}^{n+1})\}$;
    	\item $S_G^p(0,T)$ is the completion of $S_G^0(0,T)$ under the norm
    	$\|\eta\|_{S_G^p}=\{\hat{\mathbb{E}}[\sup_{t\in[0,T]}|\eta_t|^p]\}^{1/p}$.
    \end{itemize}

    We denote by $\langle B\rangle$ the quadratic variation process of the $G$-Brownian motion $B$. For two processes $\eta\in M_G^p(0,T)$ and $\zeta\in H_G^p(0,T)$, Peng established the $G$-It\^{o} integrals $\int_0^\cdot \eta_s d\langle B\rangle_s$ and $\int_0^\cdot \zeta_s dB_s$.  Similar to the classical Burkholder--Davis--Gundy inequality, the following property holds.

    \begin{proposition}[\cite{HJPS2}]\label{the1.3}
	If $\eta\in H_G^{\alpha}(0,T)$ with $\alpha\geq 1$ and $p\in(0,\alpha]$, then
	$\sup_{u\in[t,T]}|\int_t^u\eta_s dB_s|^p\in L_G^1(\Omega_T)$ and
	\begin{displaymath}
	\underline{\sigma}^p c_p\hat{\mathbb{E}}_t[(\int_t^T |\eta_s|^2ds)^{p/2}]\leq
	\hat{\mathbb{E}}_t[\sup_{u\in[t,T]}|\int_t^u\eta_s dB_s|^p]\leq
	\bar{\sigma}^p C_p\hat{\mathbb{E}}_t[(\int_t^T |\eta_s|^2ds)^{p/2}],
	\end{displaymath}
	where $0<c_p<C_p<\infty$ are constants.
    \end{proposition}



    We now introduce some basic results of $G$-BSDEs. Consider the following type of $G$-BSDE
    \begin{equation}\label{eq1.1}
    Y_t=\xi+\int_t^T f(s,Y_s,Z_s)ds+\int_t^T g(s,Y_s,Z_s)d\langle B\rangle_s-\int_t^T Z_s dB_s-(K_T-K_t),
    \end{equation}
    where
    \begin{displaymath}
    f(t,\omega,y,z), \  g(t,\omega,y,z):[0,T]\times\Omega_T\times\mathbb{R}\times\mathbb{R}\rightarrow \mathbb{R}
    \end{displaymath}
    satisfy the following properties:
    \begin{description}
    	\item[(H1)] There exists some $\beta>1$ such that for any $y, z \in \mathbb{R}$,  $f(\cdot,\cdot,y,z), \ g(\cdot,\cdot,y,z)\in M_G^{\beta}(0,T)$;
    	\item[(H2)] There exists some $L>0$ such that
    	\begin{displaymath}
    	|f(t,y,z)-f(t,y',z')|+|g(t,y,z)-g(t,y',z')|\leq L(|y-y'|+|z-z'|).
    	\end{displaymath}
    \end{description}

    For simplicity, we denote by $\mathfrak{S}_G^{\alpha}(0,T)$ the collection of processes $(Y,Z,K)$ such that $Y\in S_G^{\alpha}(0,T)$, $Z\in H_G^{\alpha}(0,T)$, and $K$ is a decreasing $G$-martingale with $K_0=0$ and $K_T\in L_G^{\alpha}(\Omega_T)$. Hu et al \cite{HJPS1,HJPS2} established the existence and uniqueness result for Equation \eqref{eq1.1} as well as the comparison theorem.

\begin{theorem}[\cite{HJPS1}]\label{the1.4}
    	Assume that $\xi\in L_G^{\beta}(\Omega_T)$ and $f,g_{ij}$ satisfy \textsc{(H1)} and \textsc{(H2)} for some $\beta>1$. Then, for any $1<\alpha<\beta$, Equation \eqref{eq1.1} has a unique solution $(Y,Z,K)\in \mathfrak{S}_G^{\alpha}(0,T)$. Moreover, we have
    	\begin{displaymath}
    	|Y_t|^\alpha\leq C\hat{\mathbb{E}}_t[|\xi|^\alpha+\int_t^T |f(s,0,0)|^\alpha+\sum_{i,j=1}^d|g_{ij}(s,0,0)|^\alpha ds],
    	\end{displaymath}
    	where the constant $C$ depends on $\alpha$, $T$, $\underline{\sigma}$ and $L$.
\end{theorem}
  Below is a generalization of Proposition 3.5 in \cite{HJPS1}.
\begin{theorem} \label{Esti-Z-GBSDE}
Let $f$ satisfy (H1) and (H2) for some $\beta>1$. Assume
\[
Y_{t}=\xi+\int_{t}^{T}f(s,Y_{s},Z_{s})ds-\int_{t}^{T}Z_{s}dB_{s}-(K_{T}
-K_{t})+(A_T-A_t),
\]
where $Y\in S_G^{\alpha}(0,T)$, $Z\in H_G^{\alpha}(0,T)$, $K, A$ are both
decreasing process with $A_0=K_{0}=0$ and $A_T, K_{T}\in L_G^{\alpha}(\Omega
_{T})$ for some $\beta\geq\alpha>1$. Then there exists a constant $C_{\alpha
}:=C(\alpha,T,\underline{\sigma},L)>0$ such that
\begin{equation*}
\mathbb{\hat{E}}[(\int_{0}^{T}|Z_{s}|^{2}ds)^{\frac{\alpha}{2}}]\leq
C_{\alpha}\bigg\{ \mathbb{\hat{E}}[\sup_{t\in\lbrack0,T]}|Y_{t}|^{\alpha
}]+\bigg(\mathbb{\hat{E}}[\sup_{t\in\lbrack0,T]}|Y_{t}|^{\alpha}]\bigg)^{\frac{1}{2}%
}\bigg(\big(\mathbb{\hat{E}}[(\int_{0}^{T}f_{s}^{0}ds)^{\alpha}]\big)^{\frac{1}{2}}+\big(m_\alpha^{A, K}\big)^{1/2}\bigg)\bigg\},
\end{equation*}
where $f_{s}^{0}=|f(s,0,0)|$, $m_\alpha^{A,K}=\min \{\hat{\mathbb{E}}[|A_T|^\alpha], \hat{\mathbb{E}}[|K_T|^\alpha]\}$.
\end{theorem}
\begin{proof}
Applying It\^{o}'s formula to $|Y_{t}|^{2}$, we have%
\[
|Y_{0}|^{2}+\int_{0}^{T}|Z_{s}|^{2}d\langle B\rangle_{s}=|\xi|^{2}+\int%
_{0}^{T}2Y_{s}f(s)ds-\int_{0}^{T}2Y_{s}Z_{s}dB_{s}-\int_{0}^{T}2Y_{s}(dK_{s}-dA_{s}),
\]
where $f(s)=f(s,Y_{s},Z_{s})$. Then%
\[
(\int_{0}^{T}|Z_{s}|^{2}d\langle B\rangle_{s})^{\frac{\alpha}{2}}\leq
C_{\alpha}\bigg\{|\xi|^{\alpha}+|\int_{0}^{T}Y_{s}f(s)ds|^{\frac{\alpha}{2}}
+|\int_{0}^{T}Y_{s}Z_{s}dB_{s}|^{\frac{\alpha}{2}}+|\int_{0}^{T}Y_{s}%
dK_{s}|^{\frac{\alpha}{2}}+|\int_{0}^{T}Y_{s}%
dA_{s}|^{\frac{\alpha}{2}}\bigg\}.
\]
By simple calculation, we can obtain%
\begin{equation}\label {Esti-Z-GBSDE-Proof}
\mathbb{\hat{E}}[(\int_{0}^{T}|Z_{s}|^{2}ds)^{\frac{\alpha}{2}}]\leq
C_{\alpha}\bigg\{ \Vert Y\Vert_{S_G^{\alpha}}^{\alpha}+\Vert Y\Vert
_{S_G^{\alpha}}^{\frac{\alpha}{2}}\bigg[(\mathbb{\hat{E}}[|K_{T}|^{\alpha
}])^{\frac{1}{2}}+(\mathbb{\hat{E}}[|A_{T}|^{\alpha
}])^{\frac{1}{2}}+(\mathbb{\hat{E}}[(\int_{0}^{T}f_{s}^{0}ds)^{\alpha
}])^{\frac{1}{2}}\bigg]\bigg\}.
\end{equation}
On the other hand, noting that
\[
K_{T}=\xi-Y_{0}+\int_{0}^{T}f(s)ds-\int_{0}^{T}Z_{s}dB_{s}+A_T,
\]
we get%
\begin{equation} \label {Esti-K-GBSDE-Proof}
\mathbb{\hat{E}}[|K_{T}|^{\alpha}]\leq C_{\alpha}\bigg\{ \Vert Y\Vert
_{S_G^{\alpha}}^{\alpha}+\mathbb{\hat{E}}[(\int_{0}^{T}|Z_{s}%
|^{2}ds)^{\alpha/2}]+\mathbb{\hat{E}}[(\int_{0}^{T}f_{s}^{0}ds)^{\alpha}]+\mathbb{\hat{E}}[|A_{T}|^{\alpha}]\bigg\}.
\end{equation}
 Suppose that $\hat{\mathbb{E}}[|K_T|^\alpha]\geq \hat{\mathbb{E}}[|A_T|^\alpha]$. By (\ref{Esti-Z-GBSDE-Proof}) and (\ref{Esti-K-GBSDE-Proof}), we have
\begin{displaymath}
	\mathbb{\hat{E}}[(\int_{0}^{T}|Z_{s}|^{2}ds)^{\frac{\alpha}{2}}]\leq
	C_{\alpha}\bigg\{ \mathbb{\hat{E}}[\sup_{t\in\lbrack0,T]}|Y_{t}|^{\alpha
	}]+\bigg(\mathbb{\hat{E}}[\sup_{t\in\lbrack0,T]}|Y_{t}|^{\alpha}]\bigg)^{\frac{1}{2}%
	}\bigg(\big(\mathbb{\hat{E}}[(\int_{0}^{T}f_{s}^{0}ds)^{\alpha}]\big)^{\frac{1}{2}}+\big(\hat{\mathbb{E}}[|A_T|^\alpha]\big)^{1/2}\bigg)\bigg\}.
\end{displaymath}
By symmetry of $K$ and $A$, we get the desired result.
\end{proof}

   \begin{theorem}[\cite{HJPS2}]\label{the1.5}
    	Let $(Y_t^l,Z_t^l,K_t^l)_{t\leq T}$, $l=1,2$, be the solutions of the following $G$-BSDEs:
    	\begin{displaymath}
    	Y^l_t=\xi^l+\int_t^T f^l(s,Y^l_s,Z^l_s)ds+\int_t^T g^l_{ij}(s,Y^l_s,Z^l_s)d\langle B^i,B^j\rangle_s+V_T^l-V_t^l-\int_t^T Z^l_s dB_s-(K^l_T-K^l_t),
    	\end{displaymath}
    	where processes $\{V_t^l\}_{0\leq t\leq T}$ are assumed to be right-continuous with left limits (RCLL), q.s., such that $\hat{\mathbb{E}}[\sup_{t\in[0,T]}|V_t^l|^\beta]<\infty$, $f^l,\ g^l_{ij}$ satisfy \textsc{(H1)} and \textsc{(H2)}, $\xi^l\in L_G^{\beta}(\Omega_T)$ with $\beta>1$. If $\xi^1\geq \xi^2$, $f^1\geq f^2$, $g^1_{ij}\geq g^2_{ij}$, for $i,j=1,\cdots,d$ and $V_t^1-V_t^2$ is an increasing process, then $Y_t^1\geq Y_t^2$.
    \end{theorem}

  Compared to the classical BSDE, there appears, in the BSDE driven by $G$-Brownian motion, an additional nonincreasing $G$-martingale $K$, which exhibits the uncertainty of the model.  The difficulty in the analysis of $G$-BSDE mainly lies in the appearance of this component. Song \cite{S16} proved that, the nonincreasing $G$-martingale could not be form of $\{\int_0^t \eta_sdt\}$ or $\{\int_0^t \gamma_sd\langle B\rangle_s\}$, where $\eta,\gamma\in M_G^1(0,T)$. More generally, he proved the following result.

    \begin{theorem}[\cite{S16}]\label{the1.6}
    	Assume that for $t\in[0,T]$, $\int_0^t \zeta_s dB_s+\int_0^t \eta_sds+K_t=L_t$, where $\zeta\in H_G^1(0,T)$, $\eta\in M_G^1(0,T)$ and $K,L$ are nonincreasing $G$-martingales. Then we have $\int_0^t \zeta_sdB_s=0$, $\int_0^t \eta_sds=0$ and $K_t=L_t$.
    \end{theorem}

    \begin{remark}\label{r1}
    	A process of the following form is called a generalized $G$-It\^{o} process:
    	\begin{displaymath}
    		u_t=u_0+\int_0^t \eta_sds+\int_0^t \zeta_s dB_s+K_t,
    	\end{displaymath}
    	where $\eta\in M_G^1(0,T)$, $\zeta\in H_G^1(0,T)$ and $K$ is a non-increasing $G$-martingale. Theorem \ref{the1.6} shows that the decomposition for generalized $G$-It\^{o} processes is unique.
    \end{remark}

	\section{$G$-BSDE with two reflection barriers}
	In this section, we give the formulation of the doubly reflected BSDE driven by $G$-Brownian motion. Particularly, the approximate Skorohod condition is introduced to guarantee the uniqueness of the solutions, which will be proved  via some a priori estimates given later.		
	\subsection{Formulation of doubly reflected BSDE driven by $G$-Brownian motion}
We formulate the doubly reflected BSDE driven by $G$-Brownian motion in details. For simplicity, we only consider the case of 1-dimensional $G$-Brownian motion. But our results and methods still hold for the case $d>1$. We are given the following data: the generators $f$ and $g$, the lower obstacle process $\{L_t\}_{t\in[0,T]}$, the upper obstacle process $\{U_t\}_{t\in[0,T]}$ and the terminal value $\xi$.

Here $f$ and $g$ are maps
	\begin{displaymath}
	f(t,\omega,y,z),g(t,\omega,y,z):[0,T]\times\Omega_T\times\mathbb{R}^2\rightarrow\mathbb{R}.
	\end{displaymath}
Below, we list the assumptions on the data of the doubly reflected $G$-BSDEs.

There exists some $\beta>2$ such that
	\begin{description}
		\item[(A1)] for any $y,z$, $f(\cdot,\cdot,y,z)$, $g(\cdot,\cdot,y,z)\in S_G^\beta(0,T)$;
		\item[(A2)] $|f(t,\omega,y,z)-f(t,\omega,y',z')|+|g(t,\omega,y,z)-g(t,\omega,y',z')|\leq \kappa(|y-y'|+|z-z'|)$ for some $\kappa>0$;
		\item[(A3)] $\{L_t\}_{t\in[0,T]}$, $\{U_t\}_{t\in[0,T]}\in S_G^\beta(0,T)$, $L_t\leq U_t$, $t\in[0,T]$, $q.s.$ and the upper obstacle is a generalized $G$-It\^{o} process of the following form
		\begin{displaymath}
		U_t=U_0+\int_0^t b(s)ds+\int_0^t \sigma(s)dB_s+K_t,
		\end{displaymath}
	where $\{b(t)\}_{t\in[0,T]},\{\sigma(t)\}_{t\in[0,T]}\in S_G^\beta(0,T)$, $K\in S_G^\beta(0,T)$ is a non-increasing $G$-martingale;
		\item[(A4)] $\xi\in L_G^\beta(\Omega_T)$ and $L_T\leq \xi\leq U_T$, $q.s.$
	\end{description}
	
	\begin{remark}
		Notice that the Assumptions (A1)-(A4) are quite similar to the ones in \cite{CK} since the non-increasing $G$-martingale $K$ is equal to $0$ when $G$ reduces to a linear function.
	\end{remark}
	
    We call a triple of processes $(Y,Z,A)$ with $Y, A\in S_G^\alpha(0,T)$, $Z\in H_G^\alpha(0,T)$, for some $2\leq\alpha\le\beta$,  a solution to the doubly reflected $G$-BSDE with the data $(\xi, f, g, L, U)$  if the following properties hold:
    \begin{description}
    	\item[(S1)] $L_t\leq Y_t\leq U_t$, $t\in[0,T]$;
	\item[(S2)] $Y_t=\xi+\int_t^T f(s,Y_s,Z_s)ds+\int_t^T g(s,Y_s,Z_s)d\langle B\rangle_s-\int_t^T Z_s dB_s+(A_T-A_t)$;	
	\item[(S3)] $(Y, A)$ satisfies  Approximate Skorohod Condition with order $\alpha$ $(\textmd{ASC}_{\alpha})$.
    \end{description}

\noindent \textbf{Condition} $(\textmd{ASC}_{\alpha})$: We say a pair of processes $(Y,  A)$ with $Y, A\in S_G^\alpha(0,T)$ satisfies the approximate Skorohod condition with order $\alpha$  (with respect to the obstacles $L, U$) if  there exist non-decreasing processes $\{A^{n,+}\}_{n\in\mathbb{N}}$, $\{A^{n,-}\}_{n\in\mathbb{N}}$ and  non-increasing $G$-martingales $\{K^n\}_{n\in\mathbb{N}}$, such that
    	\begin{itemize}
    		\item $\hat{\mathbb{E}}[|A_T^{n,+}|^\alpha+|A_T^{n,-}|^\alpha+|K^n_T|^\alpha]\leq C$, where $C$ is  independent of $n$;
    		\item $\hat{\mathbb{E}}[\sup\limits_{t\in[0,T]}|A_t-(A_t^{n,+}-A_t^{n,-}-K_t^n)|^\alpha]\rightarrow 0$, as $n\rightarrow\infty$;
    		\item  $\lim\limits_{n\rightarrow\infty}\hat{\mathbb{E}}[|\int_0^T (Y_s-L_s)d A_s^{n,+}|^{\alpha/2}]=0$;
    		\item  $\lim\limits_{n\rightarrow\infty}\hat{\mathbb{E}}[|\int_0^T (U_s-Y_s)d A_s^{n,-}|^{\alpha/2}]=0$.
    	\end{itemize}

Below is the main result of this paper, which gives the wellposedness of the doubly reflected $G$-BSDE.
	
\begin{theorem}\label{main}
		Suppose that $\xi$, $f$, $g$, $L$ and $U$ satisfy (A1)-(A4). Then the reflected $G$-BSDE with data $(\xi,f,g,L,U)$ has a unique solution $(Y,Z,A)$. Moreover, for any $2\leq \alpha<\beta$ we have $Y\in S^\alpha_G(0,T)$, $Z\in H_G^\alpha(0,T)$ and $A\in S_G^{\alpha}(0,T)$.
	\end{theorem}
    \begin{remark}

Recall that, in the classical case (see \cite{CK}),  the Skorohod condition below is required to guarantee the uniqueness of the solution $(Y, Z,  A)$ to the doubly reflected BSDE with parameters $(\xi,f,L,U)$:
	     $\int_0^T (Y_s-L_s)dA^+_s=\int_0^T (U_s-Y_s)dA^-_s=0$,  where $A^+$, $A^-$ are two non-decreasing processes and $A=A^+-A^-$.
		
Therefore, a more natural definition of the solution to the $G$-RBSDE $(\xi,f,g,L,U)$ is a triple of processes $(Y,Z, A)$ satisfying (Si), (Sii) and the following Skorohod condition.

\noindent \textbf{Condition} $(\textmd{SC})$: The process $A$ is decomposed as $A=\tilde {A}-K$ with $\tilde {A}$ a finite variation process  and  $K$ a non-increasing $G$-martingale, such that

    		\[\int_0^T (Y_s-L_s)d \tilde{A}_s^{+}=\int_0^T (U_s-Y_s)d \tilde{A}_s^{-}=0,\]
    		where $\tilde{A}^+$, $\tilde{A}^-$ are two non-decreasing processes and $A=A^+-A^-$.
  \end{remark}

Since the Skorohod condition  is stronger than the approximate  Skorohod condition, it follows from Theorem $\ref {main}$ that the solution satisfying Condition (SC) is unique. The existence of the solutions satisfying Condition (SC) is equivalent to prove the decomposition of the process $A$ in Theorem \ref {main} :
\begin {center}
$A=\tilde {A}-K,$  where $\tilde {A}$ a finite variation process satisfying the Skorohod condition and \\  $K$ a non-increasing $G$-martingale.
\end {center}
The existence and uniqueness of this decomposition are both interesting problems, which will be considered in future.

    \begin{remark}
    	Suppose that $U\equiv \infty$, i.e., the doubly reflected $G$-BSDE is reduced to the reflected $G$-BSDE with a lower obstacle. We can show that $A\in S_G^\alpha(0,T)$ is non-decreasing and satisfies the martingale condition, that is, $\{-\int_0^t (Y_s-L_s)dA_s\}_{t\in[0,T]}$ is a non-increasing $G$-martingale, which is the definition of solution to reflected $G$-BSDE with a lower obstacle (see \cite{LPSH}).
    	
     In fact,	let $\{A^{n,+}\}_{n\in\mathbb{N}}$, $\{A^{n,-}\}_{n\in\mathbb{N}}$ and $\{K^n\}_{n\in\mathbb{N}}$ be the approximation sequences for $A$. It is clear that $A^{n,-}\equiv 0$ for any $n\in\mathbb{N}$.  Note that $\{A^{n,+}-K^n\}$ is non-decreasing and
    	\begin{displaymath}
    	\lim_{n\rightarrow\infty}\hat{\mathbb{E}}[\sup_{t\in[0,T]}|A_t-(A_t^{n,+}-K_t^n)|^\alpha]=0,
    	\end{displaymath}
    	then $A$ is non-decreasing. Since $Y\leq L$ and $K^n$ is a non-increasing $G$-martingale, it follows that $\{-\int_0^t (Y_s-L_s)dK^n_s\}_{t\in[0,T]}$ is a non-increasing $G$-martingale for any $n\in\mathbb{N}$. It suffices to show that
    	\begin{displaymath}
    	\lim_{n\rightarrow\infty}\hat{\mathbb{E}}[\sup_{t\in[0,T]}|-\int_0^t (Y_s-L_s)dA_s-\int_0^T(Y_s-L_s)dK_s^n|]=0.
    	\end{displaymath}
    	It is easy to check that
    	\begin{align*}
    	&\hat{\mathbb{E}}[\sup_{t\in[0,T]}|-\int_0^t (Y_s-L_s)dA_s-\int_0^T(Y_s-L_s)dK_s^n|]\\
    	\leq &\hat{\mathbb{E}}[\sup_{t\in[0,T]}|\int_0^t (Y_s-L_s)d(A_s-\tilde{A}_s^{n})|]+\hat{\mathbb{E}}[\sup_{t\in[0,T]}|\int_0^t (Y_s-L_s)dA_s^{n,+}|],
    	\end{align*}
    	where $\tilde{A}^n=A^{n,+}-K^n$. Applying Lemma \ref{l1} below yields the desired result.
    	
    	By a similar analysis as above, if $L\equiv -\infty$, the definition of solution to doubly reflected $G$-BSDE can be reduced to the one of the upper obstacle case studied in \cite{lp}.
    \end{remark}

    	\begin{remark} For some results, we will replace the Assumptions (A1), (A3) by the following weaker ones.
    		\begin{description}
    			\item[(A1')] For any $y,z$, $f(\cdot,\cdot,y,z)$, $g(\cdot,\cdot,y,z)\in M_G^\beta(0,T)$;
    			\item[(A3')] $\{L_t\}_{t\in[0,T]}$, $\{U_t\}_{t\in[0,T]}\in S_G^\beta(0,T)$, $L_t\leq U_t$, $t\in[0,T]$, $q.s.$ and there exists a generalized $G$-It\^{o} process $I$ such that $L\le I\le U$,  where
    			\[I_t=I_0+\int_0^t b^I(s)ds+\int_0^t \sigma^I(s)dB_s+K^I_t,\]
    			with $b^I\in M_G^\beta(0,T), \sigma^I\in H_G^\beta(0,T)$, $K^I_0=0$ and  $K^I\in S_G^\beta(0,T)$ a non-increasing $G$-martingale.
    		\end{description}
    	\end{remark}

     \begin{remark}
    	 	Since the generator $g$ plays the same role as $f$, in the following of this paper, we only consider the case that $g=0$.
    \end{remark}

 \subsection{Some a priori estimates}

    In this subsetion, we give a priori estimate for the solution of the reflected $G$-BSDE, which implies the uniqueness of the solution to doubly reflected $G$-BSDE. In the following of this paper, we denote by $C$ a constant depending on $\alpha, T, \kappa,\underline{\sigma}$, but not on $n$, which may vary from line to line.

      Let us denote by $Var_0^T(A)$ the total variation of a process $A$ on $[0,T]$. We first introduce the following lemma.
    	\begin{lemma}\label{l1}
    		For $\alpha>1$, let $A$, $\{A^n\}_{n\in\mathbb{N}}\subset S_G^\alpha(0,T)$ be processes such that  $\hat{\mathbb{E}}[|Var_0^T(A^n)|^\alpha]\leq C$ and
    		\begin{displaymath}
    		\lim_{n\rightarrow\infty}\hat{\mathbb{E}}[\sup_{t\in[0,T]}|A_t-A_t^n|^\alpha]=0,
    		\end{displaymath}
    		where $C$ is independent of $n$. Then, we have $\hat{\mathbb{E}}[|Var_0^T(A)|^\alpha]\leq C$. Moreover, if $Y\in S_G^{p}(0,T)$, with $p=\frac{\alpha}{\alpha-1}$, we have \[\lim_{n\rightarrow\infty}\hat{\mathbb{E}}[\sup_{t\in[0,T]}|\int_0^t Y_sd(A_s-A_s^n)|]=0.\]
    \end{lemma}

    \begin{proof}
    	    We first show that $A$ is a finite variation process. Let
    		\begin{displaymath}
    		\mathcal{A}=\{\sum_{i=1}^{n-1} a_i I_{(t_i,t_{i+1}]}(s)||a_i|=1, 0\leq t_0<\cdots<t_n=T,n\in\mathbb{N}\}.
    		\end{displaymath}
    		Since $\sup_{t\in[0,T]}|A_t-A_t^n|$ converges to $0$ under the norm $\|\cdot\|_{L_G^1}$, we may choose a subsequence, still denoted by $A^n$, such that $\sup_{t\in[0,T]}|A_t-A_t^n|$ converges to $0$, q.s. It follows that, for any $a\in\mathcal{A}$
    		\begin{displaymath}
    		\lim_{n\rightarrow\infty}\int_0^T a(s)dA_s^n=\int_0^T a(s)dA_s.
    		\end{displaymath}
    		Then we have
    		\begin{align*}
    			Var_0^T(A)&=\sup_{a\in\mathcal{A}}\int_0^T a(s)dA_s=\sup_{a\in\mathcal{A}}\liminf_n\int_0^T a(s)dA_s^n\\
    			&\leq \liminf_n\sup_{a\in\mathcal{A}}\int_0^T a(s)dA_s^n=\liminf_n Var_0^T(A^n).
    		\end{align*}
    		Hence, it follows from the assumption that $\hat{\mathbb{E}}[|Var_0^T(A)|^\alpha]\leq C$. It remains to prove that for any $Y\in S_G^{p}(0,T)$, with $p=\frac{\alpha}{\alpha-1}$, we have \[\lim_{n\rightarrow\infty}\hat{\mathbb{E}}[\sup_{t\in[0,T]}|\int_0^t Y_sd(A_s-A_s^n)|]=0.\]	In fact, for each $m\in\mathbb{N}$, let $\widetilde{Y}^m_t=\sum_{i=0}^{m-1}Y_{t_i^m}I_{[t_i^m,t_{i+1}^m}(t)$, where $t_i^m=\frac{iT}{m}$, $i=0,1,\cdots,m$. Set
    		\begin{displaymath}
    		\textbf{I}=\sup_{t\in[0,T]}|\int_0^t \widetilde{Y}^m_s d(A_s-A_s^n)|,\  \textbf{II}=\sup_{t\in[0,T]}|\int_0^t (Y_s-\widetilde{Y}_s^n)d(A_s-A_s^n)|.
    		\end{displaymath}
    		By simple calculation, we have
    		\begin{align*}
    		\hat{\mathbb{E}}[\textbf{I}]
    		\leq& \sum_{i=0}^{m-1}\hat{\mathbb{E}}[\sup_{s\in[0,T]}|Y_s|(|A^n_{t_{i+1}^m}-A_{t_{i+1}^m}|+|A^n_{t_{i}^m}-A_{t_{i}^m}|)]\\
    		\leq &(\hat{\mathbb{E}}[\sup_{s\in[0,T]}|Y_s|^p])^{1/p}
    		\sum_{i=0}^{m-1}\{(\hat{\mathbb{E}}[|A^n_{t_{i+1}^m}-A_{t_{i+1}^m}|^\alpha])^{1/\alpha}+(\hat{\mathbb{E}}[|A^n_{t_{i}^m}-A_{t_{i}^m}|^\alpha])^{1/\alpha}\},\\
    		\hat{\mathbb{E}}[\textbf{II}]
    		\leq& (\hat{\mathbb{E}}[\sup_{s\in[0,T]}|Y_s-\widetilde{Y}_s^m|^p])^{1/p}\{
    		(\hat{\mathbb{E}}[|Var_0^T(A^n)|^\alpha])^{1/\alpha}+(\hat{\mathbb{E}}[|Var_0^T(A)|^\alpha])^{1/\alpha}\}.
    		\end{align*}
    		Letting $n$ tend to infinity yields that $\hat{\mathbb{E}}[I]\rightarrow 0$, for any $m\in\mathbb{N}$. Then, letting $m$ approach to infinity, we obtain that $\hat{\mathbb{E}}[II]\rightarrow 0$ by Lemma 3.2 in \cite{HJPS1}. The proof is complete.
    \end{proof}

	\begin{proposition}\label{lem0}
		Let $(\xi^1,f^1,L,U)$ and $(\xi^2,f^2,L,U)$ be two sets of data each one satisfying all the assumptions (A1)-(A4). Let $(Y^i,Z^i,A^i)$ be a solution of the reflected $G$-BSDE with data $(\xi^i,f^i,L,U)$, $i=1,2$, respectively. Set $\hat{Y}_t=Y^1_t-Y^2_t$,  $\hat{\xi}=\xi^1-\xi^2$. Then there exists a constant $C:=C(\alpha,T, \kappa,\underline{\sigma})>0$ such that
		\begin{displaymath}
		|\hat{Y}_t|^\alpha\leq C\hat{\mathbb{E}}_t[|\hat{\xi}|^\alpha+\int_t^T|\hat{\lambda}_s|^\alpha ds],
		\end{displaymath}
		where $\hat{\lambda}_s=|f^1(s,Y_s^2,Z_s^2)-f^2(s,Y_s^2,Z_s^2)|$.
	\end{proposition}
	
	\begin{proof}
		Set $\hat{Z}_t=Z_t^1-Z_t^2$, $\hat{A}_t=A_t^1-A_t^2$.  By the $G$-It\^{o} formula, we have
		\begin{displaymath}
			d|\hat{Y}_t|^2=-2\hat{Y}_t(f^1(t,Y_t^1,Z_t^1)-f^2(t,Y_t^2,Z_t^2))dt+2\hat{Y}_t\hat{Z}_tdB_t+\hat{Z}_t^2 d\langle B\rangle_t-2\hat{Y}_td\hat{A}_t.
		\end{displaymath}
		 For any $r>0$, applying $G$-It\^{o}'s formula to $H_t^{\alpha/2}e^{rt}=(|\hat{Y}_t|^2)^{\alpha/2} e^{rt}$, we have
		\begin{equation}\label{dr}
		\begin{split}
		&\quad H_t^{\alpha/2}e^{rt}+\int_t^T re^{rs}H_s^{\alpha/2}ds+\int_t^T \frac{\alpha}{2} e^{rs}
		H_s^{\alpha/2-1}(\hat{Z}_s)^2d\langle B\rangle_s\\
		&=|\hat{\xi}|^\alpha e^{rT}+
		\alpha(1-\frac{\alpha}{2})\int_t^Te^{rs}H_s^{\alpha/2-2}(\hat{Y}_s)^2(\hat{Z}_s)^2d\langle B\rangle_s-\int_t^T\alpha e^{rs}H_s^{\alpha/2-1}\hat{Y}_s\hat{Z}_sdB_s\\
		&\quad+\int_t^T{\alpha} e^{rs}H_s^{\alpha/2-1}\hat{Y}_s(f^1(s,Y_s^1,Z_s^1)-f^2(s,Y_s^2,Z_s^2))ds +\int_t^T\alpha e^{rs}H_s^{\alpha/2-1}\hat{Y}_sd\hat{A}_s.
		\end{split}\end{equation}
		From the assumption of $f^1$, we have
		\begin{align*}
		&\int_t^T{\alpha} e^{rs}H_s^{\alpha/2-1}\hat{Y}_s(f^1(s,Y_s^1,Z_s^1)-f^2(s,Y_s^2,Z_s^2))ds\\
		\leq &\int_t^T{\alpha}e^{rs}H_s^{\frac{\alpha-1}{2}}\{|f^1(s,Y_s^1,Z_s^1)-f^1(s,Y_s^2,Z_s^2)|+\hat{\lambda}_s\}ds\\
		\leq &\int_t^T{\alpha}e^{rs}H_s^{\frac{\alpha-1}{2}}\{\kappa(|\hat{Y}_s|+|\hat{Z}_s|)+\hat{\lambda}_s\}ds\\
		\leq &\tilde{r}\int_t^T e^{rs}H_s^{\alpha/2}ds
		+\frac{\alpha(\alpha-1)}{4}\int_t^Te^{rs}H_s^{\alpha/2-1}(\hat{Z}_s)^2d\langle B\rangle_s\\ &+ \int_t^T{\alpha} e^{rs}H_s^{\alpha/2-1/2}|\hat{\lambda}_s|ds,
		\end{align*}
		where $\tilde{r}=\alpha \kappa+\frac{\alpha \kappa^2}{\underline{\sigma}^2(\alpha-1)}$. Then by Young's inequality, we obtain
		\begin{displaymath}
		\int_t^T{\alpha} e^{rs}H_s^{\alpha/2-1/2}|\hat{\lambda}_s|ds
		\leq (\alpha-1)\int_t^T  e^{rs}H_s^{\alpha/2}ds+\int_t^T e^{rs}|\hat{\lambda}_s|^\alpha ds.
		\end{displaymath}

 Let $\{A^{i,n,+}\}_{n\in\mathbb{N}}$, $\{A^{i,n,-}\}_{n\in\mathbb{N}}$ and $\{K^{i,n}\}_{n\in\mathbb{N}}$ be the approximation sequences for $A^{i}$, $i=1,2$. Set $A^{i,n}=A^{i,n,+}-A^{i,n,-}-K^{i,n}$, $i=1,2$. It is easy to check that
\begin{align*}
	\int_t^T\alpha e^{rs}H_s^{\alpha/2-1}\hat{Y}_sdA^{1}_s=&\int_t^T\alpha e^{rs}H_s^{\alpha/2-1}\hat{Y}_sd({A}^1_s-A_s^{1,n})+\int_t^T\alpha e^{rs}H_s^{\alpha/2-1}\hat{Y}_sd{A}^{1,n}_s\\
	\leq &|\int_t^T\alpha e^{rs}H_s^{\alpha/2-1}\hat{Y}_sd({A}^1_s-A_s^{1,n})|+\int_t^T\alpha e^{rs}H_s^{\alpha/2-1}(\hat{Y}_s)^+dA_s^{1,n,+}\\
	&+\int_t^T\alpha e^{rs}H_s^{\alpha/2-1}(\hat{Y}_s)^-dA_s^{1,n,-}-\int_t^T\alpha e^{rs}H_s^{\alpha/2-1}(\hat{Y}_s)^+dK_s^{1,n}.
\end{align*}
By Lemma \ref{l1}, we have for any $t\in[0,T]$
\begin{displaymath}
	\lim_{n\rightarrow\infty}\hat{\mathbb{E}}[|\int_t^T\alpha e^{rs}H_s^{\alpha/2-1}\hat{Y}_sd({A}^1_s-A_s^{1,n})|]=0.
\end{displaymath}
Note that $Y_s^i\geq L_s$, for any $s\in[0,T]$ and $i=1,2$, which implies that $\hat{Y}_s\leq Y_s^1-L_s$. Hence, we have $(\hat{Y}_s)^+\leq Y_s^1-L_s$. By simple calculation, we obtain that
\begin{align*}
	\hat{\mathbb{E}}[\int_t^T\alpha e^{rs}H_s^{\alpha/2-1}(\hat{Y}_s)^+dA_s^{1,n,+}]&\leq C\hat{\mathbb{E}}[\sup_{t\in[0,T]}(|Y_t^1|+|Y_t^2|)^{\alpha-2}\int_t^T (\hat{Y}_s)^+d A_s^{1,n,+}]\\
	&\leq C(\hat{\mathbb{E}}[\sup_{t\in[0,T]}(|Y_t^1|^\alpha+|Y_t^2|^\alpha)])^{\frac{\alpha-2}{\alpha}}(\hat{\mathbb{E}}[|\int_t^T (\hat{Y}_s)^+d A_s^{1,n,+}|^{\frac{\alpha}{2}}])^{\frac{2}{\alpha}}.
\end{align*}
Recalling the definition of approximate Skorohod condition, we have
\begin{displaymath}
	\lim_{n\rightarrow\infty}\hat{\mathbb{E}}[|\int_t^T\alpha e^{rs}H_s^{\alpha/2-1}(\hat{Y}_s)^+dA_s^{1,n,+}|]=0.
\end{displaymath}
Similar analysis as above yields that
\begin{align*}
	&\lim_{n\rightarrow\infty}\hat{\mathbb{E}}[|\int_t^T\alpha e^{rs}H_s^{\alpha/2-1}(\hat{Y}_s)^-dA_s^{1,n,-}|]=0,\\
	&\lim_{n\rightarrow\infty}\hat{\mathbb{E}}[|\int_t^T\alpha e^{rs}H_s^{\alpha/2-1}(\hat{Y}_s)^+dA_s^{2,n,-}|]=0,\\
	&\lim_{n\rightarrow\infty}\hat{\mathbb{E}}[|\int_t^T\alpha e^{rs}H_s^{\alpha/2-1}(\hat{Y}_s)^-dA_s^{2,n,+}|]=0.
\end{align*}

Set $M^n_t=\int_0^t\alpha e^{rs}H_s^{\alpha/2-1}(\hat{Y}_s\hat{Z}_sdB_s+(\hat{Y}_s)^{+}dK^{1,n}_s+(\hat{Y}_s)^-dK_s^{2,n})$, $n\geq 1$. By Lemma 3.4 in \cite{HJPS1},  $M^n$ is a $G$-martingale. Let $r=\tilde{r}+\alpha$. Combining the above inequalities, we get
		\begin{align*}
		&H_t^{\alpha/2}e^{rt}+(M^n_T-M^n_t)\\ \leq &|\hat{\xi}|^{\alpha} e^{rT}+\int_t^T e^{rs}|\hat{\lambda}_s|^\alpha ds+\sum_{i=1}^2|\int_t^T\alpha e^{rs}H_s^{\alpha/2-1}\hat{Y}_sd({A}^i_s-A_s^{i,n})|\\
		&+\int_t^T\alpha e^{rs}H_s^{\alpha/2-1}(\hat{Y}_s)^+d(A_s^{1,n,+}+A_s^{2,n,-})
		+\int_t^T\alpha e^{rs}H_s^{\alpha/2-1}(\hat{Y}_s)^-d(A_s^{1,n,-}+A_s^{2,n,+})
		\end{align*}
		Taking conditional expectations on both sides and letting $n\rightarrow \infty$, there exists a constant $C:=C(\alpha,T, L,\underline{\sigma})>0$ such that
		\begin{displaymath}
		|\hat{Y}_t|^\alpha\leq C\hat{\mathbb{E}}_t[|\hat{\xi}|^\alpha+\int_t^T|\hat{\lambda}_s|^\alpha ds].
		\end{displaymath}
		The proof is complete.
	\end{proof}
	
	\section{ Proof of the main result}
	
	In this section, we will focus on the penalization method in order to get the existence of solutions to doubly reflected $G$-BSDEs.  For $n\in\mathbb{N}$,  consider the following family of $G$-BSDEs
	\begin{equation}\label{equ1}
	Y_t^n=\xi+\int_t^T f(s,Y_s^n,Z_s^n)ds+n\int_t^T(Y_s^n-L_s)^-ds -n\int_t^T(Y_s^n-U_s)^+ds-\int_t^T Z_s^ndB_s-(K_T^n-K_t^n).
	\end{equation}
	
	Now let $A_t^{n,-}=n\int_0^t (Y_s^n-U_s)^+ds$, $A_t^{n,+}=n\int_0^t (Y_s^n-L_s)^-ds$. Then $\{A_t^{n,\pm}\}_{t\in[0,T]}$ are nondecreasing processes. We can rewrite $G$-BSDE \eqref{equ1} as
	\begin{equation}\label{equ}
	Y_t^n=\xi+\int_t^T f(s,Y_s^n,Z_s^n)ds-\int_t^T Z_s^ndB_s-(K_T^n-K_t^n)+(A_T^{n,+}-A_t^{n,+})-(A_T^{n,-}-A_t^{n,-}).
	\end{equation}
	
\subsection {Uniform estimates of $Y^n$}

Under the weaker Assumptions (A1'), (A2), (A3'),  (A4), we show that $\{Y^n\}_{n=1}^\infty$ are uniformly bounded under the norm $\|\cdot\|_{S_G^\alpha}$.   	
\begin{lemma}\label{Esti-Y}
For $2\leq \alpha<\beta$, there exists a constant $C$ independent of $n$, such that  \[\hat{\mathbb{E}}[\sup_{t\in[0,T]}|Y_t^n|^\alpha]\leq C.\]	
\end{lemma}
	
\begin{proof} Let $I_t=I_0+\int_0^t b^I(s)ds+\int_0^t \sigma^I(s)dB_s+K^I_t$ be the generalized $G$-It\^o process such that $L\le I\le U$. 	Set
$\bar{Y}_t^n=Y_t^n-{I}_t,$  $\bar{Z}_t^n=Z_t^n-\sigma^I(t)$, $H_t=(\bar{Y}^n_t)^2$, $\bar{U}_t=U_t-I_t$, $\bar{L}_t=L_t-I_t$, and $\bar{f}_t=f(t,Y^n_t, Z^n_t)+b^I(t)$. $G$-BSDE \eqref{equ1} can be rewritten as
\begin{align*}
\bar{Y}_t^n=&\xi-I_T+\int_t^T \bar{f}(s)ds+n\int_t^T(\bar{Y}_s^n-\bar{L}_s)^-ds -n\int_t^T(\bar{Y}_s^n-\bar{U}_s)^+ds\\
 &-\int_t^T \bar{Z}_s^ndB_s-(K_T^n-K_t^n) +(K^I_T-K^I_t).
\end{align*}
For any $r>0$, applying It\^{o}'s formula to $H_t^{\alpha/2}e^{rt}$, we get
\begin{displaymath}
\begin{split}
&\quad H_t^{\alpha/2}e^{rt}+\int_t^T re^{rs}H_s^{\alpha/2}ds+\int_t^T \frac{\alpha}{2} e^{rs}H_s^{\alpha/2-1}(\bar{Z}_s^n)^2d\langle B\rangle_s\\
&=|\xi-{I}_T|^\alpha e^{rT}+\alpha(1-\frac{\alpha}{2})\int_t^Te^{rs}H_s^{\alpha/2-2}(\bar{Y}_s^n)^2(\bar{Z}_s^n)^2d\langle B\rangle_s\\
&\quad-\int_t^T\alpha e^{rs}H_s^{\alpha/2-1}n\bar{Y}_s^n(\bar{Y}_s^n-\bar{U}_s)^+ds+\int_t^T\alpha e^{rs}H_s^{\alpha/2-1}n\bar{Y}_s^n(\bar{Y}_s^n-\bar{L}_s)^-ds\\
&\quad+\int_t^T{\alpha} e^{rs}H_s^{\alpha/2-1}\bar{Y}_s^n\bar{f}_sds-\int_t^T\alpha e^{rs}H_s^{\alpha/2-1}(\bar{Y}_s^n\bar{Z}_s^ndB_s+\bar{Y}_s^ndK_s^n-\bar{Y}^n_sdK^I_s).
\end{split}
\end{displaymath}
Noting that $-\bar{Y}_s^n(\bar{Y}_s^n-\bar{U}_s)^+\le0$ and $\bar{Y}_s^n(\bar{Y}_s^n-\bar{L}_s)^-\le0$, we get	
\begin{displaymath}
\begin{split}
&\quad H_t^{\alpha/2}e^{rt}+\int_t^T re^{rs}H_s^{\alpha/2}ds+\int_t^T \frac{\alpha}{2} e^{rs}H_s^{\alpha/2-1}(\bar{Z}_s^n)^2d\langle B\rangle_s\\
&\leq|\xi-{I}_T|^\alpha e^{rT}+\alpha(1-\frac{\alpha}{2})\int_t^Te^{rs}H_s^{\alpha/2-2}(\bar{Y}_s^n)^2(\bar{Z}_s^n)^2d\langle B\rangle_s\\
&\quad+\int_t^T{\alpha} e^{rs}H_s^{\alpha/2-1/2}|\bar{f}_s|ds-(M_T-M_t),
\end{split}
\end{displaymath}	
where \[M_t=\int_0^t\alpha e^{rs}H_s^{\alpha/2-1}(\bar{Y}_s^n\bar{Z}_sdB_s+(\bar{Y}_s^n)^+dK_s^n+(\bar{Y}^n_s)^-dK^I_s)\] is a $G$-martingale.
From the assumption of $f$, we have
\begin{align*}
&\int_t^T{\alpha} e^{rs}H_s^{\alpha/2-1/2}|\bar{f}_s|ds\\
\leq &\int_t^T{\alpha} e^{rs}H_s^{\alpha/2-1/2}\{|f(s,0,0)|+|b^I(s)|+\kappa[|\bar{Y}_s^n|+|\bar{Z}_s^n|+|{I}_s|+|\sigma^I(s)|]\}ds\\
\leq &(\alpha \kappa+\frac{\alpha \kappa^2}{\underline{\sigma}^2(\alpha-1)})\int_t^T e^{rs}H_s^{\alpha/2}ds+\frac{\alpha(\alpha-1)}{4}\int_t^Te^{rs}H_s^{\alpha/2-1}(\bar{Z}_s^n)^2d\langle B\rangle_s\\
&+\int_t^T \alpha e^{rs}H_s^{\alpha/2-1/2}[|f(s,0,0)|+|b^I(s)|+\kappa(|{I}_s|+|\sigma^I(s)|)] ds.
\end{align*}
By Young's inequality, we obtain
\begin{align*}
&\int_t^T \alpha e^{rs}H_s^{\alpha/2-1/2}[|f(s,0,0)|+|b^I(s)|+\kappa(|{U}_s|+|\sigma^I(s)|)] ds\\
\leq &4(\alpha-1)\int_t^T e^{rs}H_s^{\alpha/2}ds +\int_t^T e^{rs}[|f(s,0,0)|^\alpha+|b^I(s)|^\alpha+\kappa^\alpha|{I}_s|^\alpha+\kappa^\alpha|\sigma^I(s)|^\alpha]ds.
\end{align*}
Combining the above inequalities, we get
\begin{align*}
&H_t^{\alpha/2}e^{rt}+\int_t^T (r-\tilde{\alpha})e^{rs}H_s^{\alpha/2}ds+\int_t^T \frac{\alpha(\alpha-1)}{4} e^{rs}H_s^{\alpha/2-1}(\bar{Z}_s^n)^2d\langle B\rangle_s+(M_T-M_t)\\
\leq &|\xi-{I}_T|^\alpha e^{rT}+\int_t^T e^{rs}[|f(s,0,0)|^\alpha+|b^I(s)|^\alpha+\kappa^\alpha|{I}_s|^\alpha+\kappa^\alpha|\sigma^I(s)|^\alpha]ds,
\end{align*}where $\tilde{\alpha}=4(\alpha-1)+\alpha \kappa+\frac{\alpha \kappa^2}{\underline{\sigma}^2(\alpha-1)}$. Setting $r=\tilde{\alpha}+1$ and taking conditional expectations on both sides, we derive that
\begin{displaymath}
H_t^{\alpha/2}e^{rt}\leq \hat{\mathbb{E}}_t[|\xi-{I}_T|^\alpha e^{rT}+\int_t^T e^{rs}[|f(s,0,0)|^\alpha+|b^I(s)|^\alpha+\kappa^\alpha|I_s|^\alpha+\kappa^\alpha|\sigma^I(s)|^\alpha]ds].
\end{displaymath}
Then, there exists a constant $C$ independent of $n$ such that
\begin{displaymath}
|\bar{Y}_t^n|^\alpha\leq C\hat{\mathbb{E}}_t[|\xi-{I}_T|^\alpha+\int_t^T [|f(s,0,0)|^\alpha+|b^I(s)|^\alpha+|\sigma^I(s)|^\alpha+|{I}_s|^\alpha]ds].
\end{displaymath}
Noting that $|Y_t^n|^\alpha\leq C(|\bar{Y}_t^n|^\alpha+|I_t|^\alpha)$ and applying Theorem \ref{the1.2}, we finally get the desired result.
\end{proof}

\subsection {Convergence of $(Y^n-U)^+$ and $(Y^n-L)^-$}	
Under the Assumptions (A1'), (A2), (A3'),  (A4), we show that $(Y^n-U)^+$ and $(Y^n-L)^-$  converge to $0$ under the norm $\|\cdot\|_{S_G^\alpha}$.   First, we prove a simple lemma.

\begin {lemma}\label {Esti-D} For $S\in S^{\beta}_G(0,T)$ with $\beta>1$, define $ \int_s^te^{-nu}dS_u:=e^{-nt}S_t-e^{-ns}S_s+\int_s^tn S_ue^{-nu}du,$ and set $i_n(t)=\hat{\mathbb{E}}_t[|\int_t^Te^{-n(s-t)}dS_s|^{\alpha}]$  for some $1\le \alpha<\beta$. Then, as $n\rightarrow\infty$, we have,
\[\hat{\mathbb{E}}[\sup_{t\in[0,T]}|i_n(t)|]\rightarrow 0.\]
\end {lemma}
\begin {proof}   Notice that the mappings  $D^n: S^\beta_G(0,T)\rightarrow S^1_G(0,T)$ by $D^n(S)=i_n$ are uniformly continuous with respect to $n$, i.e.,
\begin{align*}\|D^n(S)-D^n(S')\|_{S^1_G}&\le 3^{\alpha}\alpha\hat{\mathbb{E}}[\sup_{t\in[0,T]}\hat{\mathbb{E}}_t[\sup_{s\in[0,T]}|S_s-S'_s|\sup_{s\in[0,T]}|S^{\theta}_s|^{\alpha-1}]]\\
&\le 3^\alpha\alpha \bigg(\hat{\mathbb{E}}\bigg[\sup_{t\in[0,T]}\hat{\mathbb{E}}_t[\sup_{s\in[0,T]}|S_s-S'_s|^\alpha]\bigg]\bigg)^{\frac{1}{\alpha}}\bigg(\hat{\mathbb{E}}\bigg[\sup_{t\in[0,T]}\hat{\mathbb{E}}_t[\sup_{s\in[0,T]}|S^\theta_s|^\alpha]\bigg]\bigg)^{\frac{\alpha-1}{\alpha}}. 
\end{align*} where $S^\theta=\theta S+(1-\theta)S'$ for some $\theta\in[0,1]$.
By Theorem \ref{the1.2}, it suffices to prove this lemma for a dense subset of $S^\beta_G(0,T)$. For a $G$-It\^o process $S_t=S_0+\int_0^t b^S(s)ds+\int_0^t \sigma^S(s)dB_s+\int_0^tc^S(s)d\langle B\rangle_s$ with $b^S, c^S, \sigma^S\in M^0_G(0,T)$, we have
\begin{align*}|i_n(t)|&\le C_{\alpha}\bigg(\hat{\mathbb{E}}_t\bigg[\big|\int_t^Te^{-n(s-t)}(|b^S(s)|+|c^S(s)|)ds\big|^{\alpha}\bigg]+\hat{\mathbb{E}}_t\bigg[\big|\int_t^Te^{-n(s-t)}\sigma^S(s)dB_s\big|^\alpha\bigg]\bigg)\\
&\le C_{\alpha}\bigg(\frac{1}{n}\bigg)^{\alpha}\hat{\mathbb{E}}_t\bigg[\sup_{s\in[0,T]}(|b^S(s)|+|c^S(s)|)^{\alpha}\bigg]+ C_{\alpha}\bigg(\frac{1}{n}\bigg)^{\frac{\alpha}{2}}\hat{\mathbb{E}}_t\bigg[\sup_{s\in[0,T]}|\sigma^S(s)|^{\alpha}\bigg].
\end{align*}
So, we get $\hat{\mathbb{E}}[\sup_{t\in[0,T]}|i_n(t)|^\alpha]\rightarrow 0$ as $n$ goes to $\infty$.
\end {proof}

\begin{lemma}\label{Esti-U-L} Let $\tilde{Y}^n, \tilde{M}^n\in S^\alpha_G(0,T)$ and $\tilde{f}^n\in M^\alpha_G(0,T)$ for some $1<\alpha\le\beta$ satisfy
\[\tilde{Y}^n_t=\xi+\int_t^T\tilde{f}^n(s)ds+n\int_t^T(\tilde{Y}_s^n-L_s)^-ds -n\int_t^T(\tilde{Y}_s^n-U_s)^+ds-(\tilde{M}^n_T-\tilde{M}^n_t).\]
Assuming that $\tilde{M}^n$ is a martingale under a time-consistent sublinear expectation $\tilde{\mathbb{E}}$, we have
\begin {eqnarray}
\label {Esti-U}& &(\tilde{Y}_t^n-U_t)^+\leq \bigg|\mathbb{\tilde{E}}_t[\int_t^Te^{-n(s-t)}\tilde{f}^n(s)ds+\int_t^Te^{-n(s-t)}dU_s]\bigg|,\\
\label {Esti-L}& &(\tilde{Y}_t^n-L_t)^-\le \bigg|\mathbb{\tilde{E}}_t[\int_t^Te^{-n(s-t)}\tilde{f}^n(s)ds+\int_t^Te^{-n(s-t)}dL_s]\bigg|.
\end {eqnarray}
\end{lemma}

\begin{proof}
For $S\in S^{\beta}_G(0,T)$, setting $\bar{Y}^n_t=\tilde{Y}^n_t-S_t$, $\bar{U}_t=U_t-S_t$ and $\bar{L}_t=L_t-S_t$, we have
\begin{align*}
 &e^{-nt}\bar{Y}^n_t+\int_t^T e^{-ns}d\tilde{M}^n_s\\
=&e^{-nT}(\xi-S_T)+\int_t^T ne^{-ns}\bigg(\bar{Y}_s^n-(\bar{Y}_s^n-\bar{U}_s)^++(\bar{Y}^n_s-\bar{L}_s)^-\bigg)ds\\
 & +\int_t^Te^{-ns}\tilde{f}^n(s)ds+\int_t^Te^{-ns}dS_s.
\end{align*}
(1) If $S_t=U_t$, we have $\xi- S_T=\xi-U_T\le0$, and
\begin{displaymath} \label {Esti-U-proof}
\bar{Y}_s^n-(\bar{Y}_s^n-\bar{U}_s)^++(\bar{Y}^n_s-\bar{L}_s)^-=-(\tilde{Y}^n_s-U_s)^-+(\tilde{Y}^n_s-L_s)^-\le0.
\end{displaymath}
So, we have
\begin{align*}
(\tilde{Y}_t^n-U_t)^+\leq \bigg|\mathbb{\tilde{E}}_t[\int_t^Te^{-n(s-t)}\tilde{f}^n(s)ds+\int_t^Te^{-n(s-t)}dU_s]\bigg|.
\end{align*}
(2) If $S_t=L_t$, we have $\xi- S_T=\xi-L_T\ge0$, and
\begin{displaymath} \label {Esti-L-proof}
\bar{Y}_s^n-(\bar{Y}_s^n-\bar{U}_s)^++(\bar{Y}^n_s-\bar{L}_s)^-=(\tilde{Y}^n_s-L_s)^+-(\tilde{Y}^n_s-U_s)^-\ge0.
\end{displaymath}
So, we have
\begin{align*}
(\tilde{Y}_t^n-L_t)^-\le \bigg|\mathbb{\tilde{E}}_t[\int_t^Te^{-n(s-t)}\tilde{f}^n(s)ds+\int_t^Te^{-n(s-t)}dL_s]\bigg|.
\end{align*}
\end{proof}
\begin{lemma}\label{Conv-U-L}
Assume that (A1'), (A2), (A3') and (A4) hold.
As $n$ goes to $\infty$, for any $2\leq \alpha<\beta$, we have
\begin {eqnarray}\label {Conv.-U-L}
\hat{\mathbb{E}}[\sup_{t\in[0,T]}|(Y_t^n-U_t)^+|^\alpha]\rightarrow0, \  \hat{\mathbb{E}}[\sup_{t\in[0,T]}|(Y_t^n-L_t)^-|^\alpha]\rightarrow0.
\end {eqnarray}
\end{lemma}
\begin {proof}
 For each given $\varepsilon>0$, we can choose a Lipschitz function $l(\cdot)$ such that $I_{[-\varepsilon,\varepsilon]}\leq l(x)\leq I_{[-2\varepsilon,2\varepsilon]}$. Thus we have
\begin{displaymath}
f(s,Y_s^n,Z_s^n)-f(s,Y_s^n,0)=(f(s,Y_s^n, Z_s^n)-f(s, Y_s^n,0))l(Z_s^n)+a^{\varepsilon,n}_sZ_s^n=:m_s^{\varepsilon,n}+a^{\varepsilon,n}_sZ_s^n,
\end{displaymath}
where $a^{\varepsilon,n}_s=(1-l(Z_s^n))(f(s,Y_s^n, Z_s^n)-f(s, Y_s^n,0))(Z_s^n)^{-1}\in M_G^2(0,T)$ with $|a^{\varepsilon,n}_s|\leq \kappa$. It is easy to check that $|m_s^{\varepsilon,n}|\leq 2\kappa\varepsilon$. Then we can get
\begin{displaymath}
f(s, Y_s^n, Z_s^n)=f(s, Y_s^n,0)+a^{\varepsilon,n}_s Z_s^n+m_s^{\varepsilon,n}.
\end{displaymath}
Now we consider the following $G$-BSDE:
\begin{displaymath}
Y^{\varepsilon,n}_t=\xi+\int_t^T a^{\varepsilon,n}_sZ^{\varepsilon,n}_sds-\int_t^T Z^{\varepsilon,n}_sdB_s-(K^{\varepsilon,n}_T-K^{\varepsilon,n}_t).
\end{displaymath}
For each $\xi\in L_G^p(\Omega_T)$ with $p>1$, define
\begin{displaymath}
\tilde{\mathbb{E}}^{\varepsilon,n}_t[\xi]:=Y^{\varepsilon,n}_t,
\end{displaymath} which is a time-consistent sublinear expectation.
Set $\tilde{B}^{\varepsilon,n}_t=B_t-\int_0^t a^{\varepsilon,n}_s ds$. By Theorem 5.2 in \cite{HJPS2}, $\{\tilde{B}^{\varepsilon,n}_t\}$ is a $G$-Brownian motion under $\tilde{\mathbb{E}}^{\varepsilon,n}[\cdot]$.

We rewrite $G$-BSDE \eqref{equ1} as the following
\begin{eqnarray} \label {BSDE-Gisanov}
\begin {split}
Y_t^n=&\xi+\int_t^T f^{\varepsilon,n}(s)ds-\int_t^Tn(Y_s^n-U_s)^+ds+\int_t^Tn(Y_s^n-L_s)^-ds\\
 &-\int_t^T Z_s^nd\tilde{B}^{\varepsilon,n}_s-(K_T^n-K_t^n),
\end {split}
\end{eqnarray}
where $f^{\varepsilon,n}(s)=f(s, Y_s^n,0)+m^{\varepsilon,n}_s$. Since $K^n$ is a martingale under $\tilde{\mathbb{E}}^{\varepsilon,n}[\cdot]$ by Theorem 5.1 in \cite {HJPS2},   it follows from (\ref {Esti-U}) in Lemma \ref {Esti-U-L}	that
\begin{align*}
(Y_t^n-U_t)^+\leq& \bigg|\tilde{\mathbb{E}}^{\varepsilon, n}_t[\int_t^Te^{-n(s-t)}f^{\varepsilon, n}(s)ds+\int_t^Te^{-n(s-t)}dU_s]\bigg|.
\end{align*}
By Theorem \ref{the1.4}, for $2\leq \alpha<\beta$, it follows that
		\begin{equation}\begin{split} \label {Conv.-U-Proof}
		\hat{\mathbb{E}}[\sup_{t\in[0,T]}|(Y^n_t-U_t)^+|^\alpha]\leq &\hat{\mathbb{E}}\bigg[\sup_{t\in[0,T]}\bigg|\tilde{\mathbb{E}}^{\varepsilon, n}_t[\int_t^Te^{-n(s-t)}f^{\varepsilon, n}(s)ds+\int_t^Te^{-n(s-t)}dU_s]\bigg|^{\alpha}\bigg]\\
		\leq & C_{\alpha}\hat{\mathbb{E}}\bigg[\sup_{t\in[0,T]}\hat{\mathbb{E}}_t[\bigg|\int_t^Te^{-n(s-t)}f^{\varepsilon, n}(s)ds+\int_t^Te^{-n(s-t)}dU_s\bigg|^{\alpha}]\bigg],
		\end{split}\end{equation}
which converges to $0$ as $n$ goes to $\infty$ by Lemma \ref {Esti-D}. Similarly, we can prove \[\lim_{n\rightarrow\infty}\hat{\mathbb{E}}[\sup_{t\in[0,T]}|(Y_t^n-L_t)^-|^\alpha]=0.\]
\end{proof}

\subsection {Uniform estimates of $Z^n$, $K^n$, $A^{n,-}$ and $A^{n,+}$}
In this subsection, we give the uniform estimates for $Z^n$, $K^n$, $A^{n,+}$ and $A^{n,-}$ under the Assumptions (A1)-(A4).  To this end, we prove that $(Y^n-U)^+$ converges to $0$ with the explicit rate $\frac{1}{n}$, which requires that the upper obstacle to be a generalized $G$-It\^{o} process.

\begin{lemma}\label{Conv-rate-A}
For $2\leq \alpha<\beta$, there exists a constant $C$ independent of $n$, such that
\begin{displaymath}
\hat{\mathbb{E}}[\sup_{t\in[0,T]}|(Y_t^n-U_t)^+|^\alpha]\leq \frac{C}{n^\alpha}.
\end{displaymath}
\end{lemma}
\begin{proof} Now $U_t=U_0+\int_0^tb(s)ds+\int_0^t\sigma(s)dB_s+K_t$ with $b, \ \sigma\in S_G^\beta(0,T)$, and $K\in S_G^\beta(0,T)$  a non-increasing $G$-martingale. Below, we employ the notations in the proof of Lemma \ref {Conv-U-L}.

We rewrite $U_t$ as
\begin {eqnarray}\label {U-Gisanov}
U_t=U_0+\int_0^tb^{\varepsilon,n}(s)ds+\int_0^t\sigma(s)d\tilde{B}^{\varepsilon, n}_s+K_t,
\end {eqnarray}		
where $b^{\varepsilon,n}(s)=b(s)+a_s^{\varepsilon, n}\sigma(s)$. By (\ref {Conv.-U-Proof}), we have, for $2\leq \alpha<\beta$,
\begin{align*}
		\hat{\mathbb{E}}[\sup_{t\in[0,T]}|(Y^n_t-U_t)^+|^\alpha]\leq &\hat{\mathbb{E}}\bigg[\sup_{t\in[0,T]}\bigg|\tilde{\mathbb{E}}^{\varepsilon, n}_t[\int_t^Te^{-n(s-t)}f^{\varepsilon, n}(s)ds+\int_t^Te^{-n(s-t)}dU_s]\bigg|^{\alpha}\bigg]\\
		\leq & \hat{\mathbb{E}}\bigg[\sup_{t\in[0,T]}\bigg|\tilde{\mathbb{E}}^{\varepsilon, n}_t[\int_t^Te^{-n(s-t)}(|f^{\varepsilon, n}(s)|+|b^{\varepsilon,n}(s)|)ds]\bigg|^{\alpha}\bigg].
\end{align*}
By Theorem \ref{the1.4}, it follows that
		\begin{align*}
		\hat{\mathbb{E}}[\sup_{t\in[0,T]}|(Y^n_t-U_t)^+|^\alpha]\leq &\frac{1}{n^\alpha}\hat{\mathbb{E}}\bigg[\sup_{t\in[0,T]}\bigg|\tilde{\mathbb{E}}^{\varepsilon, n}_t[\sup_{s\in[0,T]}|f^{\varepsilon, n}(s)+b^{\varepsilon, n}(s)|]\bigg|^{\alpha}\bigg]\\
		\leq &C_{\alpha}\frac{1}{n^\alpha}\hat{\mathbb{E}}\bigg[\sup_{t\in[0,T]}\hat{\mathbb{E}}_t[\sup_{s\in[0,T]}|f^{\varepsilon, n}(s)+b^{\varepsilon, n}(s)|^{\alpha}]\bigg].
		\end{align*}
Since $\hat{\mathbb{E}}\bigg[\sup_{t\in[0,T]}\hat{\mathbb{E}}_t[\sup_{s\in[0,T]}|f^{\varepsilon, n}(s)+b^{\varepsilon, n}(s)|^{\alpha}]\bigg]$ are uniformly bounded, we get the desired result.\end{proof}
\begin{lemma}\label{Esti-A-K}
For $2\leq \alpha<\beta$, there exists a constant $C$ independent of $n$, such that
\begin{displaymath}
\hat{\mathbb{E}}[|K_T^n|^\alpha]\leq C, \ \hat{\mathbb{E}}[|A_T^{n,-}|^\alpha]\leq C,\  \hat{\mathbb{E}}[|A_T^{n,+}|^\alpha]\leq C, \ and \  \mathbb{\hat{E}}[(\int_{0}^{T}|Z^n_{s}|^{2}ds)^{\frac{\alpha}{2}}]\leq C.
\end{displaymath}
\end{lemma}
\begin{proof}
By Lemma \ref{Conv-rate-A}, there exists a constant $C$ independent of $n$ such that \[\hat{\mathbb{E}}[|A_T^{n,-}|^\alpha]=n^\alpha\hat{\mathbb{E}}[(\int_0^T (Y_s^n-U_s)^+ds)^\alpha]\leq C.\]
Then, it follows from Theorem \ref{Esti-Z-GBSDE} that 	$\mathbb{\hat{E}}[(\int_{0}^{T}|Z^n_{s}|^{2}ds)^{\frac{\alpha}{2}}]$ are uniformly bounded. Noting that
\begin{displaymath}
		K_T^n-A_T^{n,+}=\xi-Y_0^n+\int_0^T f(s,Y_s^n,Z_s^n)ds-\int_0^T Z_s^ndB_s+A_T^{n,-},
\end{displaymath} we conclude that $\hat{\mathbb{E}}[|K_T^n-A_T^{n,+}|^\alpha]$ are uniformly bounded. Since $K_T^n$ and $-A_T^{n,+}$ are non-positive, the proof is complete.
	\end{proof}

\subsection {Proof of Theorem \ref {main}}
In this subsection, we prove that $Y^n$, $Z^n$, and $A^n=A^{n,+}-K^n-A^{n,-}$, $n\geq 1$ are Cauchy sequences with respect to the norms $\|\cdot\|_{S^\alpha_G}$, $\|\cdot\|_{H^\alpha_G}$  and $\|\cdot\|_{S^\alpha_G}$, respectively, and that their limits are a solution to the doubly reflected $G$-BSDE.	

	\begin{lemma}\label{Conv-Y-Z-A}
		For  $2\leq \alpha<\beta$, we have
		\[\lim_{n,m\rightarrow\infty}\hat{\mathbb{E}}[\sup_{t\in[0,T]}|Y_t^n-Y_t^m|^\alpha]=0,\ \lim_{n,m\rightarrow\infty}\hat{\mathbb{E}}[(\int_0^T|Z_s^n-Z_s^m|^2ds)^{\frac{\alpha}{2}}]=0, \
		 \lim_{n,m\rightarrow\infty}\hat{\mathbb{E}}[\sup_{t\in[0,T]}|A_t^n-A_t^m|^\alpha]=0.\]
	\end{lemma}
	
	\begin{proof}
		For any $r>0$,  and $n, m\in\mathbb{N}$, set
$$
\begin{array}{lll}{\hat{Y}_t=Y_t^n-Y_t^m,} & {\hat{Z}_t=Z_t^n-Z_t^m,} & {\hat{K}_t=K_t^n-K_t^m,} \\ {\hat{A}_t^+=A_t^{n,+}-A_t^{m,+},} & {\hat{A}_t^-=A_t^{n,-}-A_t^{m,-},} & {\hat{f}_t=f(t,Y_t^n,Z_t^n)-f(t,Y_t^m,Z_t^m).} \end{array}
$$		
Denote $H_t=|\hat{Y}_t|^2$. Applying It\^{o}'s formula to $H_t^{\alpha/2}e^{rt}$, we  get
		\begin{displaymath}
		\begin{split}
		&\quad H_t^{\alpha/2}e^{rt}+\int_t^T re^{rs}H_s^{\alpha/2}ds+\int_t^T \frac{\alpha}{2} e^{rs}
		H_s^{\alpha/2-1}(\hat{Z}_s)^2d\langle B\rangle_s\\
		&=
		\alpha(1-\frac{\alpha}{2})\int_t^Te^{rs}H_s^{\alpha/2-2}(\hat{Y}_s)^2(\hat{Z}_s)^2d\langle B\rangle_s
		+\int_t^T\alpha e^{rs}H_s^{\alpha/2-1}\hat{Y}_sd(\hat{A}_s^+-\hat{A}_s^-)\\
		&\quad+\int_t^T{\alpha} e^{rs}H_s^{\alpha/2-1}\hat{Y}_s\hat{f}_sds-\int_t^T\alpha e^{rs}H_s^{\alpha/2-1}(\hat{Y}_s\hat{Z}_sdB_s+\hat{Y}_sd\hat{K}_s).
		\end{split}
		\end{displaymath}
Noting that $A_t^{n,-}=n\int_0^t (Y_s^n-U_s)^+ds$, $A_t^{n,+}=n\int_0^t (Y_s^n-L_s)^-ds$, we have
\begin{align*}
 &\int_t^T\alpha e^{rs}H_s^{\alpha/2-1}\hat{Y}_sd(\hat{A}_s^+-\hat{A}_s^-)\\
=&\int_t^T\alpha e^{rs}H_s^{\alpha/2-1}\bigg[(Y^n_s-L_s)-(Y^m_s-L_s)\bigg](dA^{n,+}_s-dA^{m,+}_s)\\
 &-\int_t^T\alpha e^{rs}H_s^{\alpha/2-1}\bigg[(Y^n_s-U_s)-(Y^m_s-U_s)\bigg](dA^{n,-}_s-dA^{m,-}_s)\\
\le& \int_t^T\alpha e^{rs}H_s^{\alpha/2-1}\bigg[(Y^n_s-L_s)^-dA^{m,+}_s+(Y^m_s-L_s)^{-}dA^{n,+}_s\bigg]\\
 &+\int_t^T\alpha e^{rs}H_s^{\alpha/2-1}\bigg[(Y^n_s-U_s)^{+}dA^{m,-}_s+(Y^m_s-U_s)^+dA^{n,-}_s\bigg]=:\int_t^T\Delta_sds.
\end{align*}
Therefore,
\begin{displaymath}
		\begin{split}
		&\quad H_t^{\alpha/2}e^{rt}+\int_t^T re^{rs}H_s^{\alpha/2}ds+\int_t^T \frac{\alpha}{2} e^{rs}
		H_s^{\alpha/2-1}(\hat{Z}_s)^2d\langle B\rangle_s\\
		&\le\alpha(1-\frac{\alpha}{2})\int_t^Te^{rs}H_s^{\alpha/2-2}(\hat{Y}_s)^2(\hat{Z}_s)^2d\langle B\rangle_s+\int_t^T{\alpha} e^{rs}H_s^{\alpha/2-1}\hat{Y}_s\hat{f}_sds\\
		&\quad+\int_t^T\Delta_sds-(M_T-M_t	),
		\end{split}
		\end{displaymath}
		where $M_t=\int_0^t \alpha e^{rs}H_s^{\alpha/2-1}(\hat{Y}_s\hat{Z}_sdB_s+(\hat{Y}_s)^+dK_s^m+(\hat{Y}_s)^-dK_s^n)$ is a $G$-martingale. Applying the H\"{o}lder inequality, we have
		\begin{displaymath}
		\int_t^T{\alpha} e^{rs}H_s^{\frac{\alpha-1}{2}}|\hat{f}_s|ds
		\leq (\alpha \kappa+\frac{\alpha \kappa^2}{\underline{\sigma}^2(\alpha-1)})\int_t^T e^{rs}H_s^{\alpha/2}ds
		+\frac{\alpha(\alpha-1)}{4}\int_t^Te^{rs}H_s^{\alpha/2-1}(\hat{Z}_s)^2d\langle B\rangle_s.
		\end{displaymath}
		Letting $r=1+\alpha \kappa+\frac{\alpha \kappa^2}{\underline{\sigma}^2(\alpha-1)}$, we have
		\begin{align*}
		H_t^{\alpha/2}e^{rt}+(M_T-M_t)\leq	\int_t^T\Delta_sds.		
		\end{align*}
		Taking conditional expectation on both sides of the above inequality, it follows that
		\begin{displaymath}\label{eq1.5}
		H_t^{\alpha/2}e^{rt}\leq
		\hat{\mathbb{E}}_t[\int_t^T \Delta_sds].
		\end{displaymath}
		Consequently, we have
		\begin{equation}\label{e1.6}
		\hat{\mathbb{E}}[\sup_{t\in[0,T]}|\hat{Y}_t|^\alpha]\leq \hat{\mathbb{E}}[\sup_{t\in[0,T]}\hat{\mathbb{E}}_t[\int_0^T \Delta_sds]].
		\end{equation}
		By symmetry and Theorem \ref{the1.2}, it suffices to prove that there exists some $\gamma>1$, such that
		\begin{equation}
		\lim_{n,m\rightarrow \infty}\hat{\mathbb{E}}[(\int_0^T H_s^{\alpha/2-1}(Y_s^n-L_s)^-dA^{m,+}_s)^{\gamma}]=0.
		\end{equation}
For $1<\gamma<\beta/\alpha$,
\begin{align*}
 &\hat{\mathbb{E}}[(\int_0^T H_s^{\alpha/2-1}(Y_s^n-L_s)^-dA^{m,+}_s)^{\gamma}]\\
\le&\hat{\mathbb{E}}[\sup_{s\in[0,T]}|\hat{Y}_s|^{(\alpha-2)\gamma}\sup_{s\in[0,T]}\big((Y_s^n-L_s)^-\big)^{\gamma}\big(A^{m,+}_T\big)^{\gamma}]\\
\le&\bigg(\hat{\mathbb{E}}[\sup_{s\in[0,T]}|\hat{Y}_s|^{\alpha\gamma}]\bigg)^{\frac{\alpha}{\alpha-2}}\bigg(\hat{\mathbb{E}}[\sup_{s\in[0,T]}\big((Y_s^n-L_s)^-\big)^{\alpha\gamma}]	\bigg)^{\frac{1}{\alpha}}\hat{\mathbb{E}}[\big(A^{m,+}_T\big)^{\alpha\gamma}]^{\frac{1}{\alpha}}	,
\end{align*} which converges to $0$ as $n$ goes to $\infty$ by Lemma \ref{Esti-Y}, Lemma \ref {Conv-U-L} and Lemma \ref{Esti-A-K}.
		
 By a similar analysis as in the proof of Theorem 5.1 in \cite{LPSH}, for some $2\leq \alpha<\beta$, we get that
		\begin{align*}
		&\hat{\mathbb{E}}[(\int_0^T |\hat{Z}_t|^2dt)^{\frac{\alpha}{2}}]\leq C\{\hat{\mathbb{E}}[\sup_{t\in[0,T]}|\hat{Y}_t|^\alpha]+(\hat{\mathbb{E}}[\sup_{t\in[0,T]}|\hat{Y}_t|^\alpha])^{1/2}\},\\
		&\hat{\mathbb{E}}[\sup_{t\in[0,T]}|\hat{A}_t|^\alpha]\leq C\{\hat{\mathbb{E}}[\sup_{t\in[0,T]}|\hat{Y}_t|^\alpha]+\hat{\mathbb{E}}[(\int_0^T |\hat{Z}_t|^2dt)^{\frac{\alpha}{2}}]\}.
		\end{align*}
The proof is complete.
\end{proof}	

Now, we prove our main result.

\begin{proof}[Proof of Theorem \ref {main}]
	
		First we prove the uniqueness. Suppose that $(Y^i,Z^i,A^i)$, $i=1,2$ are solutions of the reflected $G$-BSDE with data $(\xi,f,,L,U)$. Proposition \ref{lem0} yields that $Y^1\equiv Y^2$. Applying $G$-It\^{o}'s formula to $(Y_t^1-Y_t^2)^2\equiv0$ and taking expectation (we may refer to Equation \eqref{dr}), we get
		\begin{displaymath}
		\hat{\mathbb{E}}[(\int_0^T |Z_s^1-Z_s^2|^2d\langle B\rangle_s)^{\alpha/2}]=0.
		\end{displaymath}
		It follows that $Z^1\equiv Z^2$. Then it is easy to check $A^1\equiv A^2$.
		
		Now we are in a position to show the existence. By Lemma \ref{Conv-Y-Z-A}, there exist $Y\in S_G^\alpha(0,T)$, $Z\in H_G^\alpha(0,T)$ and a finite variation process $A\in S_G^\alpha(0,T)$ such that
		\begin{displaymath}
		\hat{\mathbb{E}}[\sup_{t\in[0,T]}|Y_t^n-Y_t|^\alpha]\rightarrow 0,\ \hat{\mathbb{E}}[(\int_0^T |Z_t^n-Z_t|^2ds)^{\frac{\alpha}{2}}]\rightarrow 0,\ \hat{\mathbb{E}}[\sup_{t\in[0,T]}|A_t^n-A_t|^\alpha]\rightarrow 0, \textrm{ as }n\rightarrow\infty,
		\end{displaymath}
		where $A_t^n=A_t^{n,+}-K_t^n-A_t^{n,-}$. By Lemma \ref{Conv-U-L} , we derive that $L_t\leq Y_t\leq U_t$, for any $t\in[0,T]$. It remains to show that $A$ satisfies the approximate Skorohod condition with order $\alpha$. We claim that $\{A^{n,+}\}_{n\in\mathbb{N}}$, $\{A^{n,-}\}_{n\in\mathbb{N}}$ and $\{K^{n}\}_{n\in\mathbb{N}}$ are  the approximation sequences. It is sufficent to prove that
        \begin{displaymath}
        	\lim_{n\rightarrow\infty}\hat{\mathbb{E}}[|\int_0^T (Y_s-L_s)dA_s^{n,+}|^{\alpha/2}]=0.
        \end{displaymath}
        In fact, it is easy to check that
        \begin{displaymath}
        	\int_0^T (Y_s-L_s)dA_s^{n,+}=\int_0^T (Y_s-Y_s^n)dA_s^{n,+}+\int_0^T (Y_s^n-L_s)n(Y_s^n-L_s)^-ds\leq \sup_{t\in[0,T]}|Y_s-Y_s^n|A_T^{n,+}.
        \end{displaymath}
    It follows that
\begin{displaymath}
	\hat{\mathbb{E}}[|\int_0^T(Y_s-L_s)dA_s^{n,+}|^{\alpha/2}]\leq (\hat{\mathbb{E}}[\sup_{t\in[0,T]}|Y_t-Y_t^n|^\alpha])^{1/2}(\hat{\mathbb{E}}[|A_T^{n,+}|])^{1/2}.
\end{displaymath}
Hence, the claim holds.
 The proof is complete.
 \end{proof}

\begin{remark}
		The analysis for the penalization method above can also be used for the single obstacle case, which will extend the results in \cite{LPSH} and \cite{lp} to a more general setting. More precisely, suppose   that  $L\in S_G^\beta(0,T)$ is bounded from above by some generalized $G$-It\^{o} process $I$ satisfying (A3'). Then the reflected $G$-BSDE with a lower obstacle whose parameters are given by $(\xi,f,g,L)$ admits a unique solution,  where $(f,g)$ satisfies (A1'), (A2) and $\xi\in L^\beta_G(\Omega_T)$ such that $\xi\geq L_T$. The reflected $G$-BSDE with an upper obstacle whose parameters are given by $(\xi,f,g,U)$ admits a unique solution, where $(f,g,U)$ satisfies (A1)-(A3) and $\xi\in L^\beta_G(\Omega_T)$ with $\xi\leq U_T$.  
\end{remark}

By the construction via penalization, we obtain the following comparison theorem for doubly reflected $G$-BSDEs.

\begin{theorem}
		Let $(\xi^i,f^i,L^i,U^i)$ be two sets of data satisfying (A1)-(A4), $i=1,2$. We furthermore assume the following:
		\begin{description}
			\item[(i)] $\xi^1\leq \xi^2$, $q.s.$;
			\item[(ii)] $f^1(t,y,z)\leq f^2(t,y,z)$, $\forall (y,z)\in\mathbb{R}^2$;
			\item[(iii)] $L_t^1\leq L^2_t$, $U_t^1\leq U_t^2$, $0\leq t\leq T$, $q.s.$
		\end{description}
		Let $(Y^i,Z^i,A^i)$ be the solutions of the doubly reflected $G$-BSDE with data $(\xi^i,f^i,L^i,U^i)$, $i=1,2$, respectively. Then
		\begin{displaymath}
		Y_t^1\leq Y^2_t, \quad 0\leq t\leq T  \quad q.s.
		\end{displaymath}
\end{theorem}

\begin{proof}
	For $i=1,2$, consider the following $G$-BSDEs parameterized by $n=1,2,\cdots$,
	\begin{align*}
	Y_t^{i,n}=&\xi^{i,n}+\int_t^T f^i(s,Y_s^{i,n},Z_s^{i,n})ds-n\int_t^T (Y^{i,n}_s-U_s^i)^+ds+\int_t^T n(Y_s^{i,n}-L_s^i)^-ds\\ &-\int_t^T Z_s^{i,n}dB_s-(K_T^{i,n}-K_t^{i,n}).
	\end{align*}
	By Theorem \ref{the1.5}, for any $t\in[0,T]$ and $n=1,2,\cdots$, we have  $Y^{1,n}_t\leq Y^{2,n}_t$. Letting $n$ go to infinity, we get the desired result.
\end{proof}

\end{document}